\documentclass[12pt]{article}
\usepackage{amsmath}
\usepackage{amssymb}
\usepackage{physics}
\usepackage{amsthm}
\usepackage{graphicx}
\usepackage{xcolor}
\usepackage{a4wide}
\usepackage{hyperref}
\usepackage{tikz}

\hypersetup{
  colorlinks   = true, 
  urlcolor     = blue, 
  linkcolor    = red,  
  citecolor    = red   
}

\graphicspath{{./figures}}

\newtheorem{lemma}{Lemma}[section]
\newtheorem{proposition}[lemma]{Proposition}
\newtheorem{remark}[lemma]{Remark}

\newtheorem{theorem}{Theorem}
\newtheorem{definition}[lemma]{Definition}
\newtheorem{corollary}[lemma]{Corollary}
\newtheorem{conjecture}[lemma]{Conjecture}

\newtheorem{question}[lemma]{Question}

\newcommand{\CC}{\mathcal{C}}
\newcommand{\OO}{\mathcal{O}}

\usepackage[backend=biber,
            style=numeric,      
            giveninits=true,
            doi=false,
            url=false,
            sorting=nyt,
            eprint=true]{biblatex}

\renewbibmacro{in:}{}

\author{Lael Edwards-Costa\thanks{
Department of Mathematics,
Pennsylvania State University,
University Park, PA 16802,
USA;
lael.costa@psu.edu}
}
\title{Outer length billiards on polygons}

\begin{document}
\maketitle

\abstract{The classical inner and outer billiards can be formulated in
    variational terms, with length and area as the respective generating
    functions. The other two combinations, ``inner with area'' and ``outer with
    length,'' are more recently described. Here, we consider the
    latter system in the special case of polygonal tables. We describe the
    behavior of orbits far away from the table and pose conjectures regarding an
    escaping orbit when the table is a square. This paper is meant to complement
    \cite{smooth} which handles smooth tables with positive curvature.
}

\section{Introduction}

\subsection{Familiar billiard systems}

Fix a convex body $K$ in the plane. We recall the familiar definition of the Birkhoff
billiard in $K$. If $x$ and $y$ are points on the boundary of $K$ such that
$\partial K$ has a tangent line at $y$, then we choose $z\in\partial K$ such
that the chords $xy$ and $yz$ make equal and reflected angles with the tangent
line at $y$. The billiard map can be written down in several different
coordinates, but geometrically it is the map which sends $xy$ to $yz$.

The outer billiard is a dynamical system on points in $\mathbb{R}^2\setminus K$.
Given a point $x$ outside $K$, there are exactly two support lines to $K$ which
pass through $x$. Let $\ell$ be the support line to $K$ passing through $x$,
such that $K$ is on the left when moving from $x$ towards $K$. If $\ell$
intersects with $K$ at a single point $p$, we construct the outer billiard map
by reflecting $x$ through $p$. The resulting point $y$ is the image of $x$ under
the outer billiard map.

\begin{remark}
    In general, outer billiards in the plane has no immediate connection to its
    inner counterpart. However, the two systems are dual on the sphere
    \cite[Chapter 9]{tab_book}.
\end{remark}

\subsection{Variational formulations}

The Birkhoff billiard's reflection law can be replaced with the following
variational formulation. An inscribed broken line connecting boundary
points $x_0,\dots,x_n$ in $K$ is a trajectory of the billiard map if and only if
its length is locally extremized among inscribed broken lines with endpoints
$x_0$ and $x_n$.. In other words, if
\begin{equation}
    \label{eq:extremized}
    L:=\sum_{i=0}^{n-1}\abs{x_{i+1}-x_i},
\end{equation}
then $\pdv{L}{x_i}=0$ exactly if $x_{i+1}=T(x_i)$ for $i=1,\dots,n-1$.
Likewise, outer billiard
trajectories are circumscribed broken lines with extremal algebraic area. Proofs
of these facts can be found in \cite{tab_book}.

We have addressed ``inner with length'' and ``outer with area,'' but what about
the other two combinations?

The literature contains several results concerning the inner billiard which
extremizes area. It is called the \emph{symplectic inner billiard} and is
constructed as follows: given $x,y\in\partial K$ such that $\partial K$ has a
tangent line at $y$. We construct a line $\ell$ parallel to the tangent line at
$y$ passing through $x$. Generically, $\ell$ intersects $\partial K$ at two
points. One of them is $z$; label the second $z$. The symplectic billiard map
takes the chord $xy$ to the chord $yz$. See \cite{symplectic} for an
introduction to this system and \cite{polysymp, pairs} for results in the
polygonal case. 

This leaves only ``outer with length,'' which is the topic of the present
article. This system was first described in \cite{twist}, and a number of
analytical results were obtained in \cite{bbf}. The quartet of billiards is also
treated in \cite{bialy}.

We now give its construction.

\subsection{Main definition}
\label{s:construction}

With $K$ convex, we wish to define a map $T:\mathbb{R}^2\setminus
K\to\mathbb{R}^2\setminus K$ such that a broken line $x_0,\dots,x_n$ is a
trajectory of $T$ if and only if its length is locally extremal in an analogous
manner to (\ref{eq:extremized}) above.

\begin{definition}
    Fix $K$ convex and a point $x$ outside of $K$ such that $x$ does not lie on
    the clockwise extension of a line segment contained in the boundary of $K$.
    There are two support lines to $K$ passing through $x$; label them $\ell_1$
    and $\ell_2$ so that traveling to $x$ from $K$ along $\ell_1$ keeps $K$ on
    the left (as shown in Figure \ref{fig:construction}). There is a unique
    circle $C$ tangent to both lines such that the point of tangency with
    $\ell_2$ coincides with $\ell_2$'s intersection with $K$. This circle has
    three mutual support lines with $K$: $\ell_1$ and $\ell_2$, as well as a
    third which we label $\ell_3$. We define $y=T(x)$ to be the point
    $\ell_2\cap\ell_3$.
\end{definition}

\begin{figure}
    \centering
    \includegraphics[width=0.6\textwidth]{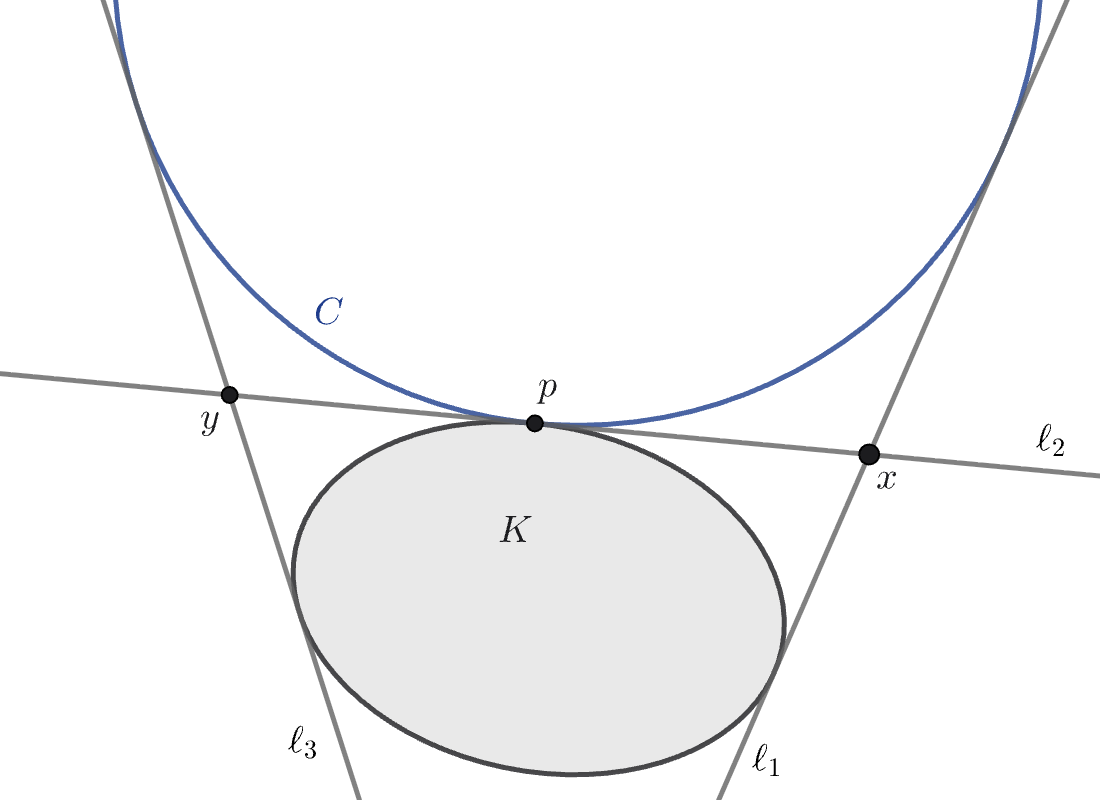}
    \caption{The construction of the outer length billiard about the convex body
    $K$.}
    \label{fig:construction}
\end{figure}

Lemma 3.1 of \cite{twist} verifies that this definition meets the requirements.

\subsection{Structure of the paper}
In Section \ref{sec:basics}, we lay out some fundamental tools in the study of
the outer length billiard around polygons. In Section \ref{sec:infinity}, we
give a proof of a remarkably universal fact about the limiting behavior of the
map. In Section \ref{sec:centers} the centers of the auxiliary circles are
considered as a dynamical system in their own right and a surprising connection
to the usual outer billiard is explored. Finally, in Section
\ref{sec:experiment} we present experimental results. In particular, Section
\ref{sec:square} presents compelling evidence for the existence of an escaping
orbit about the square.

\paragraph*{Acknowledgements}
The author would like to thank his doctoral advisor, Serge Tabachnikov, for his
support throughout this project. He is also grateful to Peter Albers for many
productive discussions on these and related topics, and to the Universit{\"a}t
Heidelberg for its generous hospitality. The author also thanks Mariia Kiyashko
for giving him a hint regarding one of the lemmas.

\section{Basic analysis}
\label{sec:basics}

\subsection{Singularity diagrams}

Let $Q$ be a convex $n$-gon. If $x$ lies on a ray which extends clockwise along
one of the sides of $Q$, there is no choice of auxiliary circle which makes $T$
continuous at $x$. We thus choose not to define the map here, or at any of its
iterated preimages. This reduces the domain of the map $T$ by a set of measure
0. We can numerically approximate this set up to a given number of iterations,
producing images of the type shown in Figure \ref{fig:tri_sing}.

\begin{figure}
    \centering
    \includegraphics[width=0.48\textwidth]{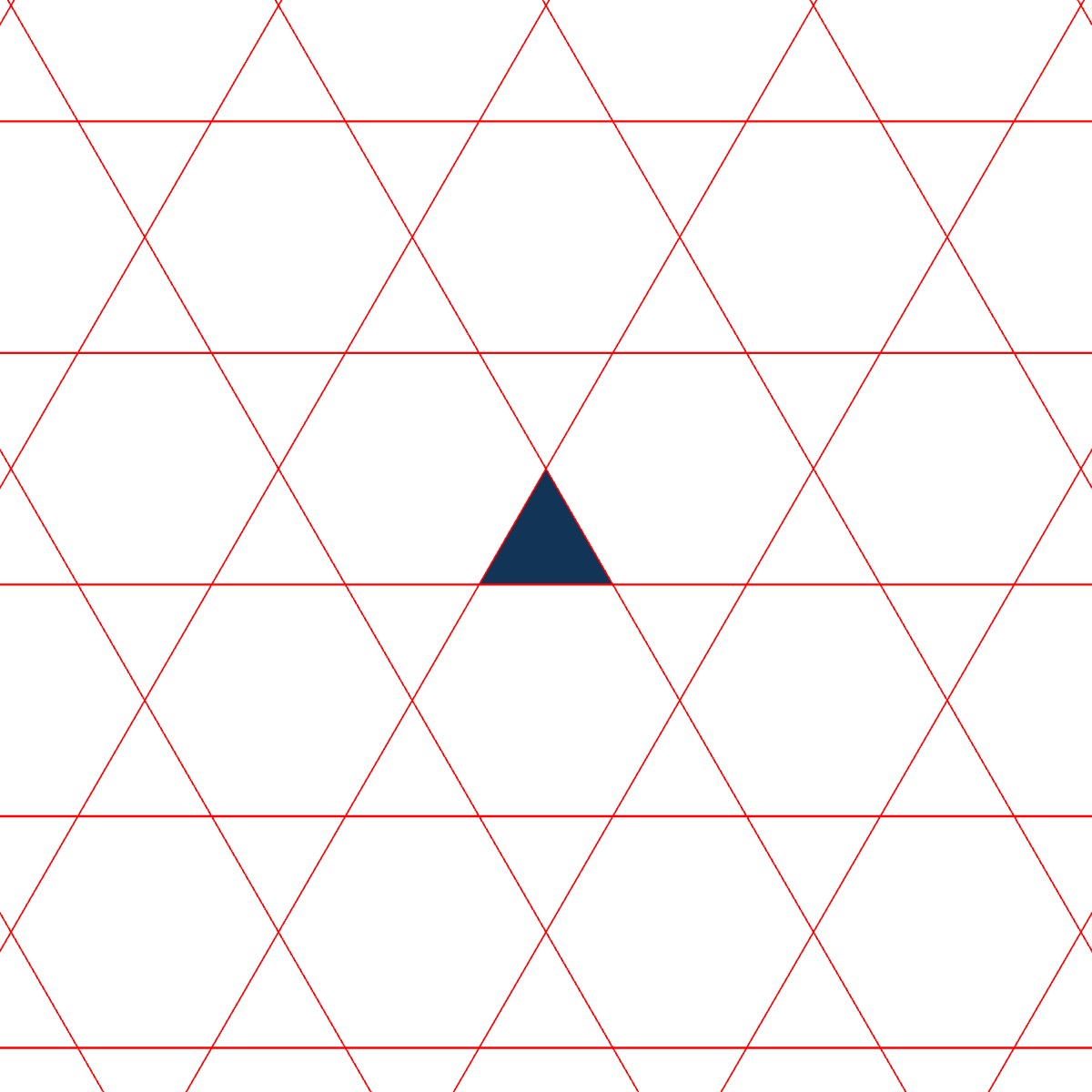}
    \hfill
    \includegraphics[width=0.48\textwidth]{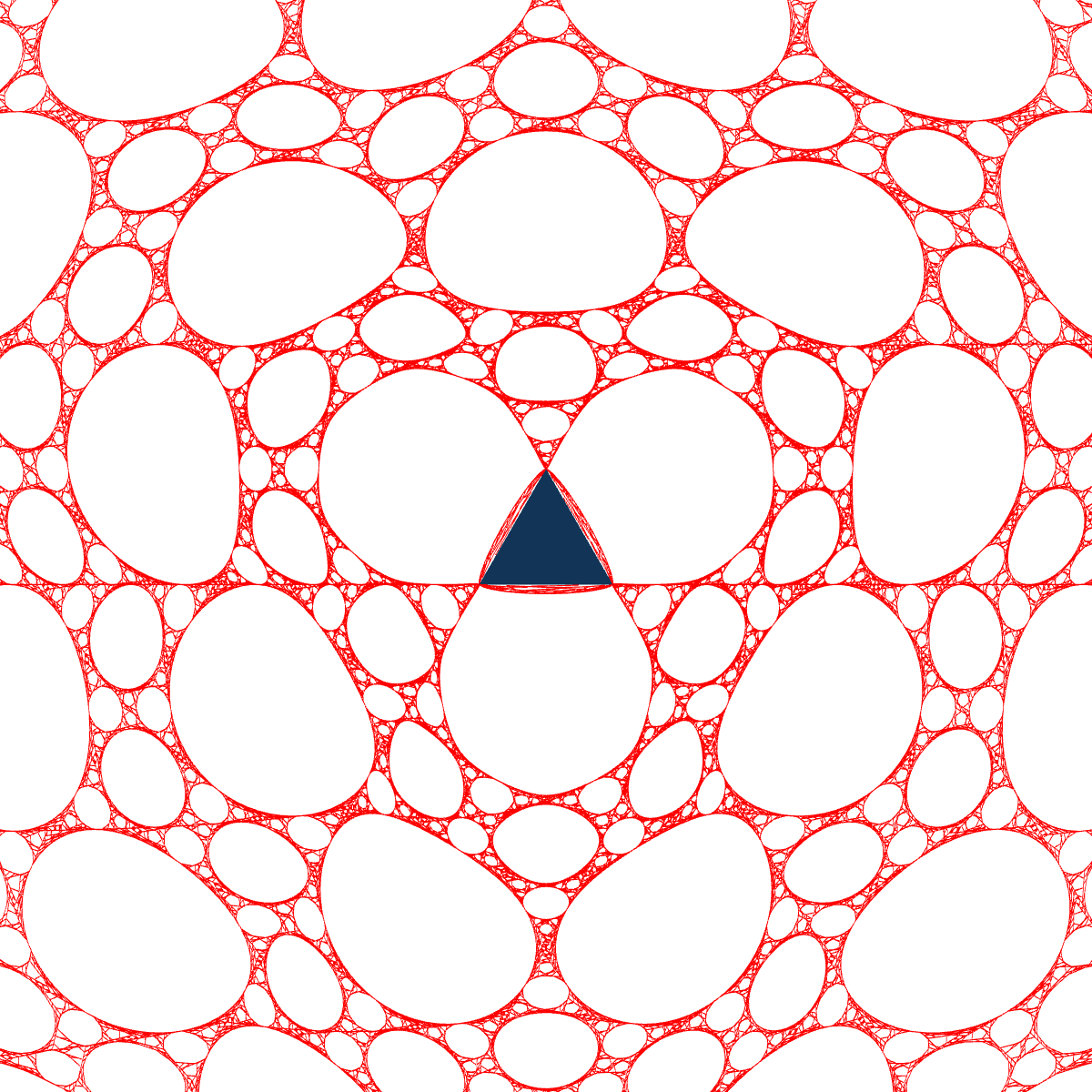}
    \caption{The singularity portrait of the equilateral triangle under the
        outer area
    (left) and outer length (right) billiards.}
    \label{fig:tri_sing}
\end{figure}

For any non-singular point, the three support lines used in the construction of
$T$ pass through vertices of $Q$. All points in a connected component of the
domain of $T$ ``see'' the same vertices and the map is continuous on each
component.

\subsection{The 2-gon}
\label{sec:2-gon}

We can fully analyze the outer length billiard map when the table is a line
segment. In this case, the set of singularities consists of the line containing
the segment 

\begin{figure}
    \centering
    \includegraphics[width=0.4\textwidth]{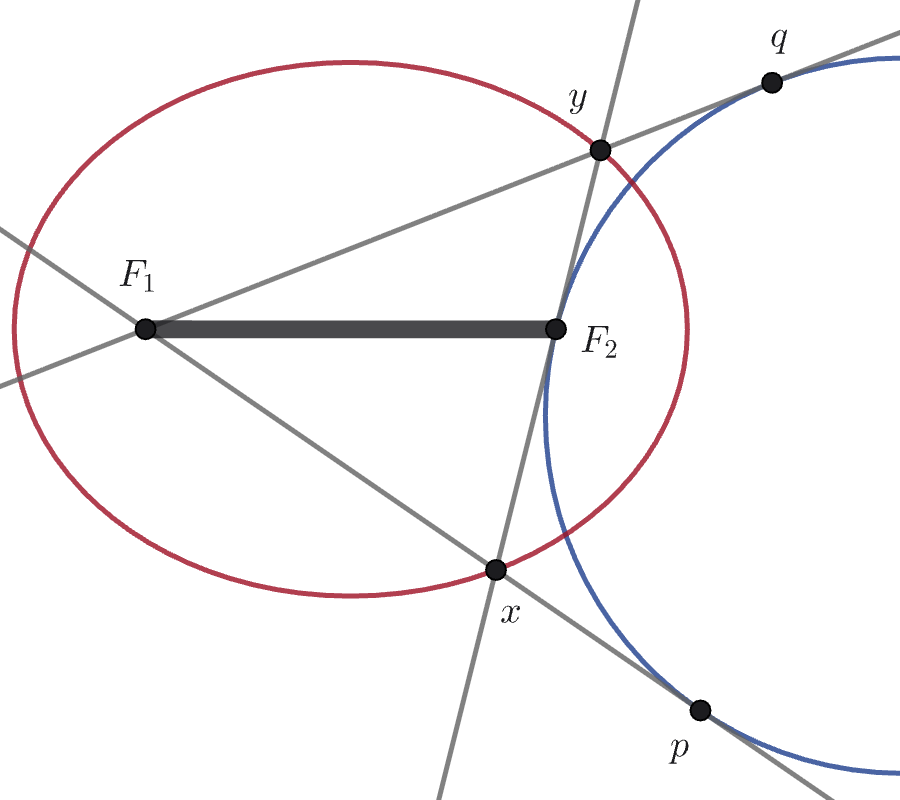}
    \caption{The construction of the outer length billiard about a 2-gon and the
    ellipse guaranteed by Proposition \ref{prop:2-gon}.}
    \label{fig:2gon}
\end{figure}

\begin{proposition}
    \label{prop:2-gon}
    Let $Q$ be a 2-gon and let $x$ be a point not collinear with $Q$. Then $x$
    and $T(x)$ lie on an ellipse $E$ with foci at the endpoints of $Q$. The
    orbit under $T$ can be identified with a usual (inner) billiard orbit in
    $E$, and in particular, the forward and reverse images are attracted to
    the major axis of $E$.
\end{proposition}

\begin{proof}
    Consider Figure \ref{fig:2gon}, which shows the construction of $y:=T(x)$.
    Repeatedly using the fact that the two segments incident to a point and
    tangent to a circle have the same length, we have that
    \begin{align*}
        \abs{F_1x}+\abs{xF_2}
        &=\abs{F_1x}+\abs{xp}\\
        &=\abs{F_1p}=\abs{F_1q}\\
        &=\abs{F_1y}+\abs{yq}\\
        &=\abs{F_1y}+\abs{yF_2},
    \end{align*}
    which gives the first statement of the proposition. The remaining statements
    follow from the well known fact that the two transfocal chords incident to a
    given point on an ellipse are related by the usual inner billiard
    transformation in that ellipse (see, for instance, \cite{tab_book}).
\end{proof}




\subsection{Triangles}
\label{sec:triangle}

The singularity portrait in Figure \ref{fig:tri_sing} (right) shows that even
the equilateral triangle gives rise to complicated dynamics. In attempting to
analyze this situation, we stumbled into a general fact about the geometry of
triangles.

Recall the following terminology from Euclidean geometry \cite{extouch}.

\begin{definition}
    Let $T=\triangle XYZ$. Consider the three lines extending the sides of $T$.
    There are three circles which are tangent to all three lines and do not
    overlap $T$. These circles are called the \emph{extouch circles} of $T$ and
    are drawn in blue in Figure \ref{fig:extouch}. The points of tangency
    between the circles and $T$ form a triangle $\triangle ABC$, which is called
    the \emph{extouch triangle} of $T$.
\end{definition}

\begin{figure}
    \centering
    \includegraphics[width=0.4\textwidth]{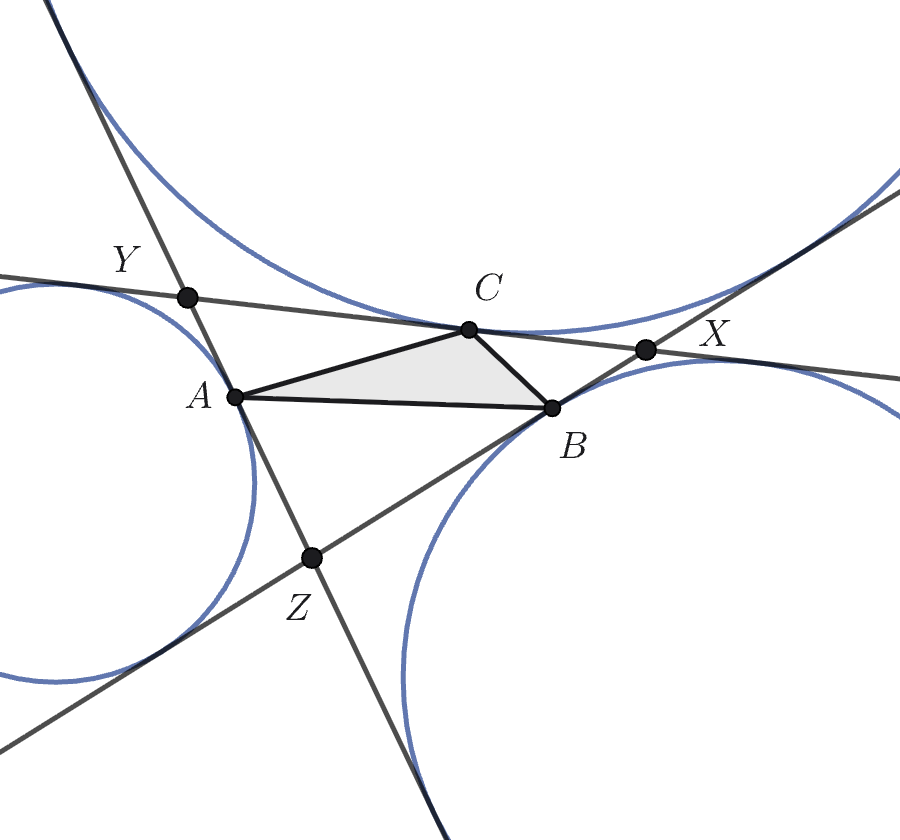}
    \caption{The triangle $\triangle ABC$ is the extouch triangle of $\triangle
    XYZ$.}
    \label{fig:extouch}
\end{figure}

Considering Figure \ref{fig:extouch}, it is clear that $\triangle ABC$ is the
extouch triangle of $\triangle XYZ$ exactly when $\triangle XYZ$ is a 3-periodic
orbit of the outer length billiard about $\triangle ABC$. 

\begin{theorem}
    \label{thm:extouch}
    Let $T$ be a triangle. Then there is a triangle $R$ such that $T$ is
    the extouch triangle of $R$.
\end{theorem}

\begin{corollary}
    Every triangular table admits a 3-periodic orbit under the outer
    length billiard.
\end{corollary}

\begin{proof}[Proof of Theorem \ref{thm:extouch}]
    Fix $T=\triangle ABC$ and let $a$, $b$, and $c$ be the lengths of the sides
    of $T$ opposite vertices $A$, $B$, and $C$, respectively, and let $|T|$
    denote the area of $T$.

    If there is a triangle $R=\triangle XYZ$ such that $T$ is the extouch
    triangle of $R$, then $R$ must satisfy the following relations
    \cite{extouch}:

    \begin{gather}
        \label{eq:sl}
        a^2=x^2-\frac{|R|^2}{yz},\quad
        b^2=y^2-\frac{|R|^2}{zx},\quad
        c^2=z^2-\frac{|R|^2}{xy}.\\
    \end{gather}

    We can rearrange (\ref{eq:sl}) to obtain
    \begin{equation}
        \label{eq:Ryzxa}
        |R|^2=yz(x^2-a^2)=zx(y^2-b^2)=xy(z^2-c^2),
    \end{equation}
    from which we can solve for $y$ and $z$ in terms of $x$:
    \begin{align*}
        y(x)&=\frac{x^2-a^2+\sqrt{x^4+2(2b^2-a^2)x^2+a^4}}{2x},\\
        z(x)&=\frac{x^2-a^2+\sqrt{x^4+2(2c^2-a^2)x^2+a^4}}{2x}.\\
    \end{align*}

    Using Heron's formula, we see that the first equality in (\ref{eq:Ryzxa})
    holds if and only if the function
    $$f(x)=yz(x^2-a^2)-S(S-x)(S-y)(S-z)$$
    (where $S(x)=\frac{1}{2}(x+y(x)+z(x))$ and $y$, $z$, and $S$ are all treated
    as functions of $x$) has a real root.

    Notice that $f$ is continuous in $x$, and that if $x=a$, then $y(x)=b$ and
    $z(x)=c$, so $f(a)=-|T|<0$. Notice also that
    $$\lim_{x\to\infty}\frac{y(x)}{x}=\lim_{x\to\infty}\frac{z(x)}{x}=1,
    \quad\text{so}\quad\lim_{x\to\infty}\frac{S(x)}{x}=\frac{3}{2}.$$
    For large $x$, the first term is asymptotically equal to $x^4-a^2x^2$, while
    the second tends towards
    $$\frac{3}{2}x\left(\frac{3}{2}x-x\right)^3=\frac{3}{16}x^4.$$
    Therefore, $\lim_{x\to\infty}f(x)=+\infty$. We conclude by the
    Intermediate Value Theorem that $f$ must have a zero at some $x_0>a$.

    The triple $(x=x_0, y=y(x_0), z=z(x_0))$ satisfies (\ref{eq:sl}) by
    construction. A triangle with these side lengths has extouch triangle
    congruent to $T$.
\end{proof}


\begin{remark}
    The author is not aware of a constructive proof, and not for want of trying. 
\end{remark}

\section{Behavior at infinity}
\label{sec:infinity}

\subsection{The outer area billiard}
\label{sec:inf_area}

Let us briefly review the behavior of orbits of the usual outer billiard far
from the table. This subsection's ideas are detailed in \cite{act}.

If $K$ is a convex body in the plane, there is a centrally symmetric convex body
$K_s$ with the property that $K$ and $K_s$ have the same width when viewed from
any angle; $K_s$ can be recovered via the Minkowski difference:
$K_s=\frac{1}{2}(K-K)$. 

Denote by $h_K(\theta)$ the (signed) distance from the origin to the support
line to $K$ which has outward normal in direction $\theta$. If a body $K$
contains the origin, we can construct its symplectic polar dual:

\begin{definition}
    \label{defn:dual}
    The symplectic polar dual to $K\in\mathbb{R}^2$ is denoted $K^*$ and is the
    body whose boundary is defined by the radial function
    $\Gamma(\theta)=\frac{1}{h_K(\theta+\pi/2)}$.
\end{definition}

It should be noted that this definition is a particular case of a construction
which works in a general symplectic space.

Distant orbits of the outer area billiard map around $K$ approximately trace a
curve which is a homothetic to $\partial(K_s)^*$ (this is a special case of the
main result of \cite{act}).

\subsection{The outer length billiard}

We now present a comparable result for the outer length billiard.
First, an observation: the outer length billiard map differs from the outer area
billiard map in a uniformly bounded manner. To phrase this precisely, we restate
\cite[Lemma 4.5]{smooth}:

\begin{lemma}
    \label{lem:uniform}
    Let $x$, $y$, and $p$ be as in Figure \ref{fig:construction}. Then
    $||xp|-|py||\leq\frac{3}{2}d$, where $d$ is the diameter of the body $K$.
\end{lemma}

\begin{remark}
    The result of the preceding lemma can be strengthened slightly: one can
    prove that the bound $\frac{3}{2}d$ can be replaced with $d$ and that this
    improved bound is sharp. However, the proof is highly technical and the
    sharper result offers no advantage in the analysis which follows. 
\end{remark}

The main result of this section is similar in nature to \cite[Theorem
2]{smooth}, but here we deal with polygonal tables instead of smooth tables with
positive curvature.
We make an observation and a definition. If $x$ is far enough
away from the table (say, more than twice the table diameter), then $T^2(x)$ is
close to $x$ and furthermore its angular displacement from $x$ is in the
clockwise direction. This motivates the following definition.

\begin{definition}
    Let $K$ be a convex body, and let $\Omega$ be the set of non-singular points
    with respect to the outer length billiard around $K$.
    The once-around orbit of a point $x\in\Omega$ under an outer billiard map is the set
    $\OO(x):=\{x,T(x),\dots,T^{2N}(x)\}$, where $T^{2N}(x)$ is the first iteration of
    the squared map which ``wraps around.'' That is, if we choose polar
    coordinates such that the point $x$ is on the $\theta=0$ radial, then $N$ is
    the smallest positive integer such that $T^{2(N-1)}(x)$ has positive
    argument while $T^{2N}(x)$ has non-positive argument (where
    $\theta\in[-\pi,\pi)$). If no such $N$ exists, the once-around orbit is
    simply the infinite forward orbit of $x$.
\end{definition}

\begin{theorem}
    \label{thm:circle}
    Let $Q$ be a convex $n$-gon containing the origin, and let $T$ denote the
    outer length billiard map about $Q$. Then there are constants $C_1$ and
    $C_2$ depending on $Q$ such that if $x$ is a point in the domain of $T$ with
    $R:=\abs{x}\geq C_2$, then all points in the once-around orbit $\OO(x)$ fall inside
    the annulus $|r-R| \leq C_1$.
\end{theorem}

The argument hinges on the following observation. Fixing
a convex table $K$ (not necessarily a polygon) and a point $x$, the construction
of the outer length billiard map proceeds by drawing two lines of support, a
circle, and then a third line. We notice that there is a line segment (see
Figure \ref{fig:virtual}) such that if $K$ were replaced by that segment and $x$
were unchanged, the construction would be unchanged, meaning that the circle,
all three support lines, and the image $T(x)$ would be the same.

This line segment, which we refer to as the \emph{virtual table} of the point
$x$, has one endpoint on $\partial K$ and the other either on $\partial K$ or
outside of $K$. In the case that $K$ is a polygon, one endpoint is always a
vertex of $K$ and the other either lies outside or is also a vertex (see Figure
\ref{fig:virtual}).

\begin{figure}
    \centering
    \includegraphics[width=0.48\textwidth]{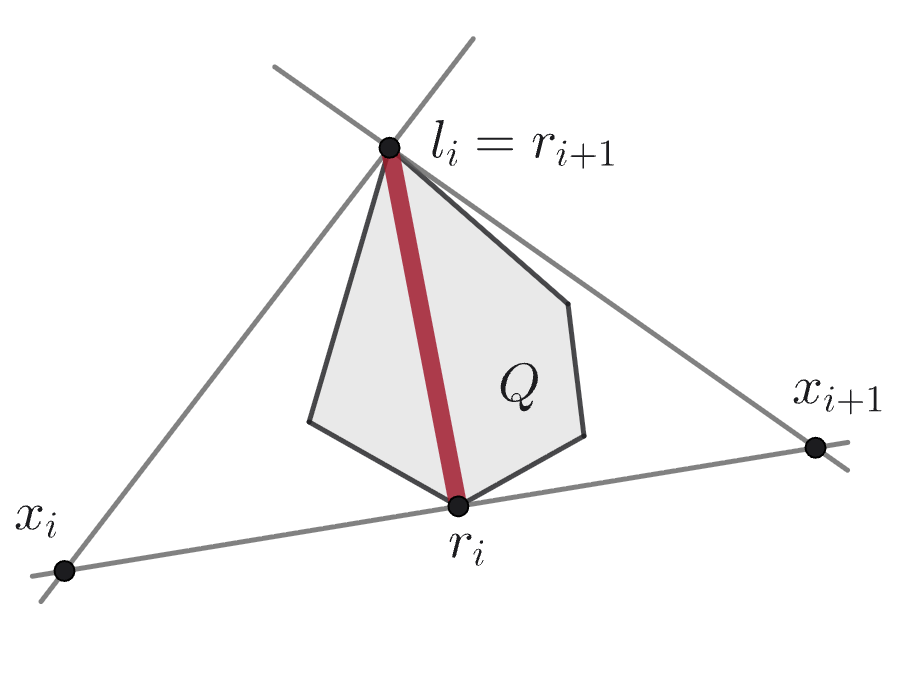}
    \hfill
    \includegraphics[width=0.48\textwidth]{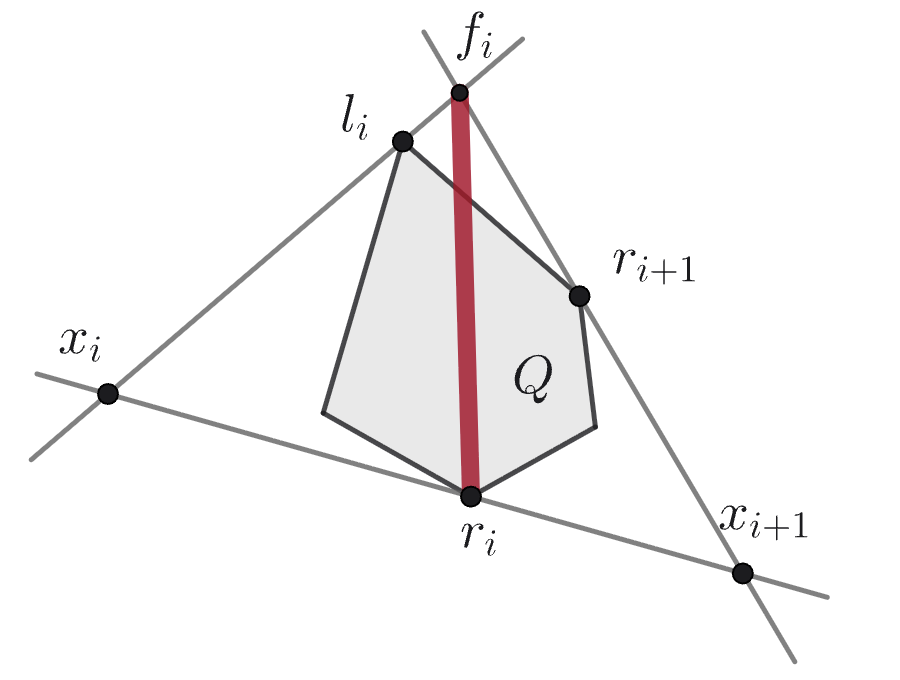}
    \caption{The virtual table is shown as a thick red line segment for a steady
    (left) and an unsteady (right) point.}
    \label{fig:virtual}
\end{figure}

By the discussion in Section \ref{sec:2-gon}, orbits around a line segment trace
out ellipses. In the polygonal case, we argue that the once-around orbit is made
up of phases of motion along nearly-circular ellipses centered close to the
origin. The transitions between these phases involve jumps from ellipse to
ellipse. Our task is to count up radial deviations from two sources:
\begin{enumerate}
    \item the eccentricity of each ellipse, and
    \item the jumps between consecutive ellipses.
\end{enumerate}

Now, let $Q$ be an $n$-gon of diameter $d$ containing the origin. Let
$V=\{v_1,\dots,v_n=v_0\}$ be the set of vertices of $Q$ with indices understood
cyclically and in counter-clockwise order, and let $\Omega$ be the set of
non-singular points with respect to $Q$. Define maps $L,R:\Omega\to V$ as
follows: through any point $x\in\Omega$ pass two support lines of $Q$; as viewed
from $x$, one of these lies to the left of $Q$. $L(x)$ is defined to be the
vertex through which this line passes. $R(x)$ is likewise the vertex of $Q$
through which the other line passes.

For the remainder of the section, let a non-singular point $x$ be fixed and let
$\OO=\OO(x)$ be the once-around orbit of $x$. Set the notation $x_i:=T^i(x)$,
$l_i:=L(x_i)$, and $r_i:=R(x_i)$. It is immediate that $r_i=l_{i+1}$ for every
$i$, but for each $i$, we have either
\begin{enumerate}
    \item $l_i=r_{i+1}$, in which case the line segment (and diagonal of $Q$)
        $l_ir_i$ is the virtual table of $x_i$, or
    \item $l_i\neq r_{i+1}$, in which case there is another point $f_i$ which
        serves as the endpoint of the virtual table.
\end{enumerate}
Let us refer to the point $x_i$ as ``steady'' if it falls into the first case
and ``unsteady'' otherwise; see Figure \ref{fig:virtual} for examples.

    \begin{lemma}
        \label{lem:obtuse}
        With notation as in Figure \ref{fig:virtual} (right), if $|x_i|\geq 5d$,
        then $\angle x_if_ix_{i+1}\geq\frac{2\pi}{3}.$
    \end{lemma}

    \begin{proof}
        Let $\alpha=\angle l_ix_ir_i$ and $\beta=\angle r_ix_{i+1}r_{i+1}$.
        By Lemma \ref{lem:uniform} and the triangle inequality, we see that both
        $|x_ir_i|$ and $|r_ix_{i+1}|$ must exceed $2d$. By the definition of $d$ we
        have $|l_ir_i|\leq d$ and $|r_ir_{i+1}|\leq d$.
        Applying the Law of Sines to $\triangle l_ix_ir_i$, we have
        $$\frac{\sin\alpha}{|l_ir_i|}=\frac{\sin\angle x_il_ir_i}{|x_ir_i|}$$
        so $\sin\alpha\leq\frac{1}{2}$. Thus $\alpha\leq \pi/6$. A similar argument
        gives that $\beta\leq\pi/6$. Finally, we must have that $\angle
        x_if_ix_{i+1} =\pi-\alpha-\beta=\geq 2\pi/3.$    
    \end{proof}

    \begin{lemma}
        \label{lem:crossings}
        If $|x_i|\geq 8d$, then $x_i$ is unsteady if and only if the
        segment $x_ix_{i+2}$ intersects a side extension of $Q$.
    \end{lemma}

    \begin{proof}
        The assumption on $|x_i|$, together with Lemmas \ref{lem:uniform} and
        \ref{lem:obtuse} ensures that $\angle x_i O x_{i+1}$ is acute, and
        because of the orientation choice made in the construction of $T$,
        $\overrightarrow{x_ix_{i+2}}$ is negatively oriented with respect to the
        origin.

        $x_i$ is unsteady exactly when either $l_i\neq l_{i+1}$ or
        $r_i\neq r_{i+1}$. This means that $x_i$ ``sees'' a different pair of
        vertices to $x_{i+2}$, which can only happen if the two points lie on
        different sides of a side extension.
    \end{proof}

\begin{figure}
    \centering
    \includegraphics[width=0.5\textwidth]{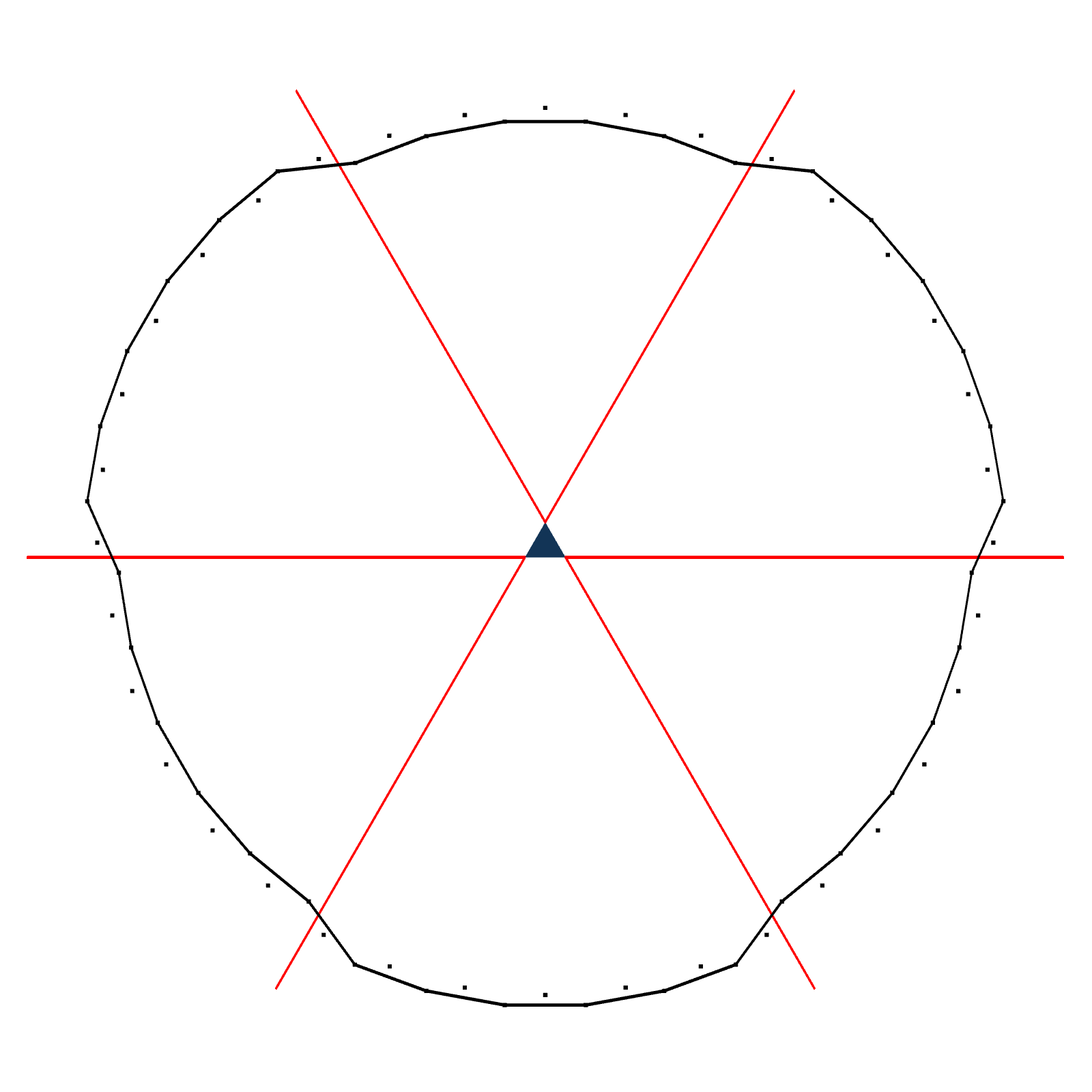}
    \caption{A once-around orbit around a triangle showing jumps across side
        extensions at unsteady points. Even points are connected by line
        segments while odd points are drawn as dots.}
        \label{fig:triangle_jump}
\end{figure}

    See Figure \ref{fig:triangle_jump} for an example.
    
    \begin{corollary}
        \label{cor:unsteady_count}
        There are at most $2n+1$ unsteady points in $\OO(x)$ and at most $2n$
        steady phases.
    \end{corollary}

    \begin{proof}
        An unsteady point occurs exactly when the line segment $x_ix_{i+2}$
        intersects a side extension of $Q$. There are exactly $2n$ of these, and
        since the orbit of the square map proceeds in a clockwise direction
        around the table, it crosses each of these once, except perhaps for the
        first one to appear clockwise from the point $x$ if $x_{2(N-1)}x_{2N}$
        intersects it.
    \end{proof}

By the preceding lemma, we have that the once-around orbit is made up of several steady
phases (in which the orbit proceeds as the usual inner billiard around a
nearly-circular ellipse), punctuated by jumps from ellipse to ellipse at a bounded number of unsteady points.

Thus, we must account for radial deviation from two sources: the eccentricity of the ellipses and the misalignment of the ellipses.

\begin{figure}
    \centering
    \includegraphics[width=0.4\textwidth]{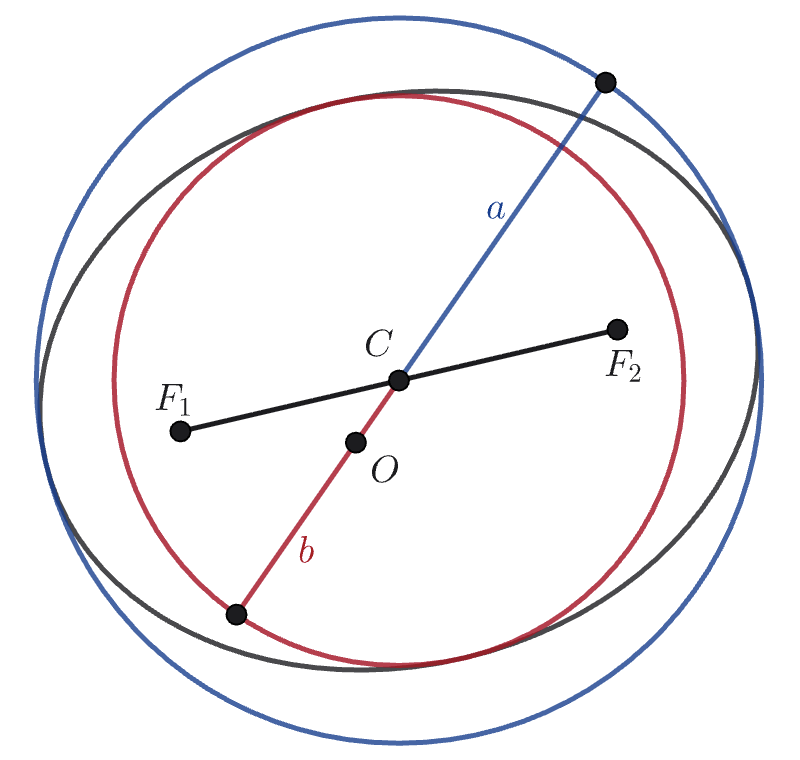}
    \caption{The configuration considered in Lemma \ref{lem:eccentricity}.}
    \label{fig:eccentricity_lemma}
\end{figure}

\begin{lemma}
    \label{lem:eccentricity}
    Let $E$ be an ellipse with foci $F_1$ and $F_2$ and center $C$ such that
    $\abs{F_1F_2}\leq 2d$ and $\abs{C}\leq 2d$ (see Figure \ref{fig:foci}). If
    $m=\min_{x\in E}\abs{x}$ and
    $M=\max_{x\in E}\abs{x}$, then $M-m\leq 5d$.
\end{lemma}

\begin{proof}
    Let $a$ and $b$ denote the semi-major and semi-minor axes of $E$,
    respectively. By the assumption that $\abs{F_1F_2}\leq 2d$ and the triangle
    inequality, we have $a - b\leq d$. Now, find that $M\leq 2d + a$ while
    $m\geq b - 2d$. Combining, we have
    \begin{align*}
        M-m &\leq 2d + a - (b - 2d) \\
            &= 4d + (a - b) \\
            &\leq 5d.
    \end{align*}
\end{proof}

\begin{figure}
    \centering
    \includegraphics[width=0.4\textwidth]{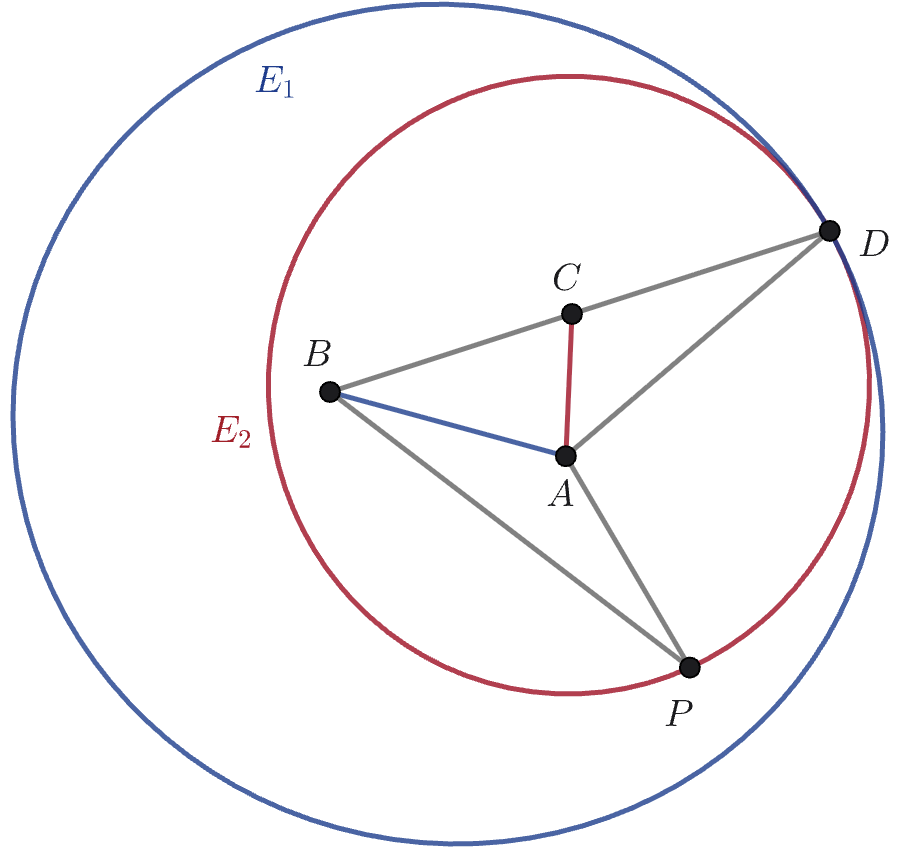}
    \caption{The configuration in Lemma \ref{lem:foci} and its proof.}
    \label{fig:foci}
\end{figure}

\begin{lemma}
    \label{lem:foci}
    Fix points $A$, $B$, $C$, and $D$ in the plane such that $B$, $C$, and $D$
    are collinear in that order. Let $E_1$ be the ellipse with foci $A$ and $B$
    passing through $D$ and let $E_2$ be the ellipse with foci $A$ and $C$
    passing through $D$ (see Figure \ref{fig:foci}). Then
    \begin{enumerate}
        \item $E_1$ and $E_2$ are tangent at $D$,
        \item $E_1$ contains $E_2$.
    \end{enumerate}
\end{lemma}

\begin{proof}
    The first statement follows from the fact that the transfocal chords
    incident to $D$ in both ellipses must be in inner billiards correspondence.
    These chords coincide, meaning that the tangent lines to the two ellipses at
    $D$ must also coincide.

    For the second statement, let $P$ be a point on $E_2$. We want to show that it
    lies inside $E_1$, i.e., that $|AP|+|BP|\leq|AD|+|BD|$. By the triangle
    inequality,
    \begin{align*}
        |BP|&\leq |CP|+|BC|\\
        |AP|+|BP|&\leq|AP|+|CP|+|BC|\\
                 &=|AD|+|CD|+|BC|\\
                 &=|AD|+|BD|,
    \end{align*}
    where the first equality follows from the string construction definition of
    $E_2$ and the second from the fact that $B$, $C$, and $D$ are collinear in
    that order.
\end{proof}

\begin{lemma}
    \label{lem:confocal}
    Let $E_1$ and $E_2$ be confocal ellipses with semi-major axes $a_1$ and
    $a_2$, respectively. If both ellipses have eccentricity at most
    $\frac{1}{2}$, then the Hausdorff distance satisfies
    $d_H(E_1,E_2)\leq2|a_2-a_1|$.
\end{lemma}

\begin{figure}
    \centering
    \includegraphics[width=0.3\textwidth]{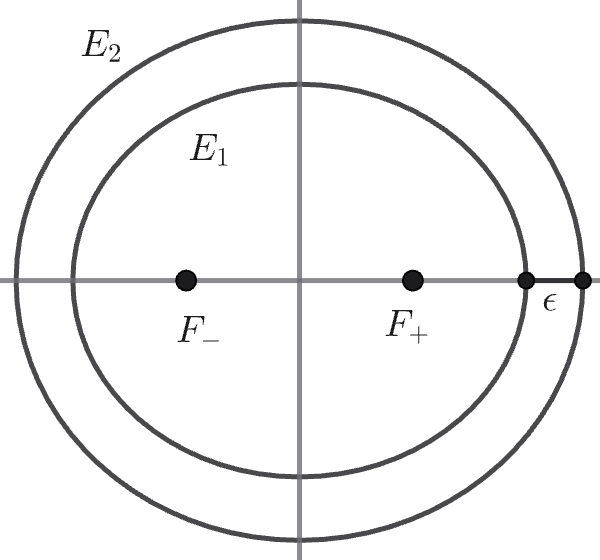}
    \caption{Confocal ellipses as in Lemma \ref{lem:confocal}.}
    \label{fig:confocal}
\end{figure}

\begin{proof}
    Fix the shared foci at $F_\pm=(\pm f,0)$ and, without loss of generality,
    say $a_2>a_1\geq2f$ so that the eccentricity assumption is satisfied (see
    Figure \ref{fig:confocal}). The
    ellipses' semi-minor axes are $b_i=\sqrt{a_i^2-f^2}$.

    The ellipse $E_i$ can be thought of as the $2a_i$ level set of the function
    $f(P)=|PF_-|+|PF_+|$. $\grad f_P$ is sum of the unit vectors in the
    directions of $\overrightarrow{F_-P}$ and $\overrightarrow{F_+P}$, so
    $\norm{\grad f}$ is minimized when the angle between these vectors is
    maximized. This occurs, for any given level set, along the $y$-axis. The
    gradient field is orthogonal to the foliation by level sets, so distance
    between two level sets is maximized by the path of smallest gradient
    vectors. Therefore, we can bound $d_H(E_1,E_2)$ by their separation along
    the $y$-axis, which is $b_2-b_1$.

    Set $\epsilon=a_2-a_1$, and let
    $g(\epsilon)=b_2-b_1=\sqrt{(a_1+\epsilon)^2-f^2}-\sqrt{a_1^2-f^2}$.
    Clearly $g(0)=0$ and
    $$g'(\epsilon)=\frac{a_1+\epsilon}{\sqrt{(a_1+\epsilon)^2-f^2}}
    =\frac{a_2}{b_2}\leq2,$$
    where the inequality follows by the assumption on eccentricity. Finally, we
    see by the racetrack principle that $d_H(E_1,E_2)\leq g(\epsilon)\leq
    2\epsilon=2|a_2-a_1|$, as desired.
\end{proof}

\begin{figure}
    \centering
    \includegraphics[width=0.6\textwidth]{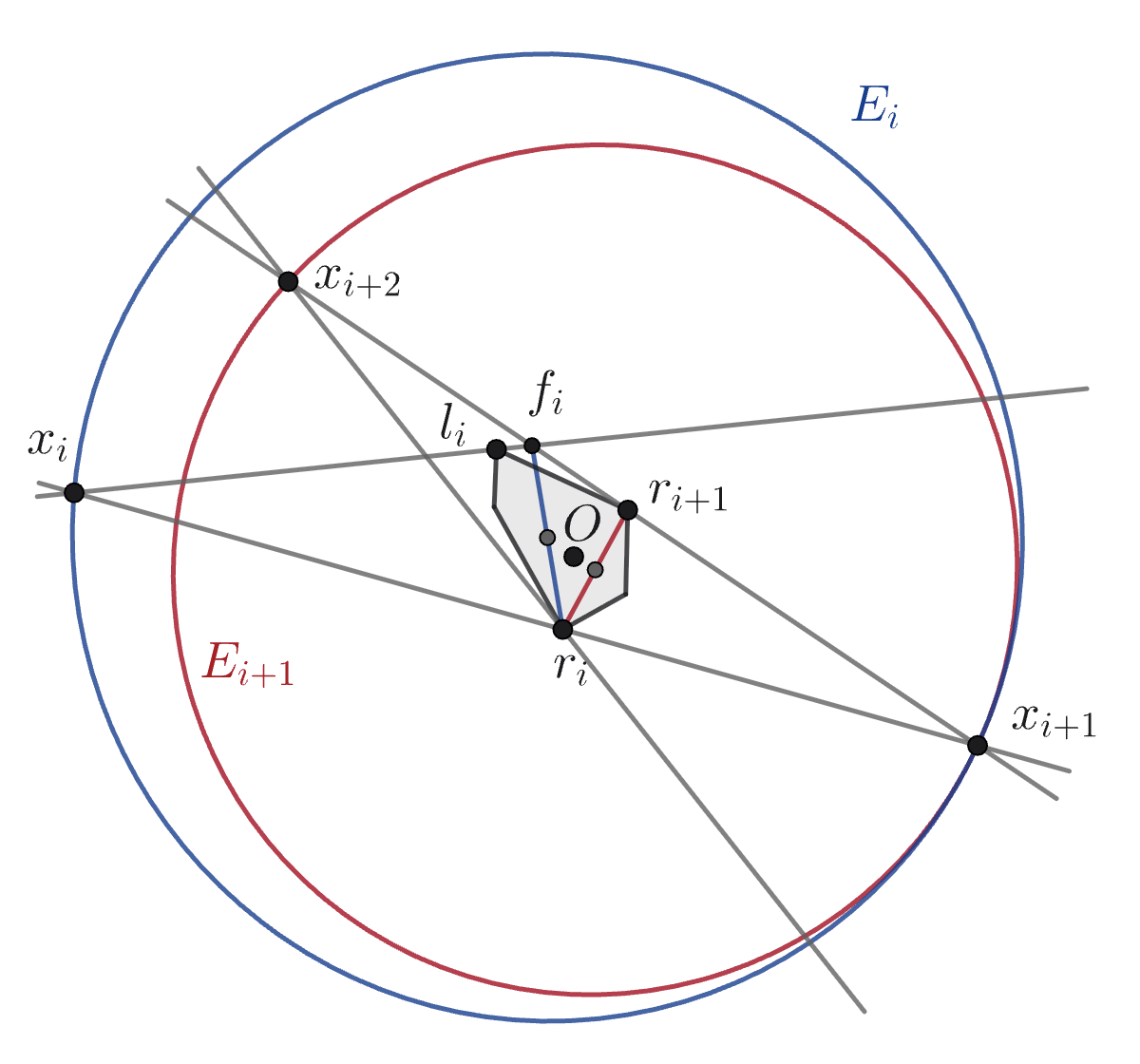}
    \caption{A jump from one ellipse to another.}
    \label{fig:jump}
\end{figure}

\begin{lemma}
    \label{lem:unsteady}
    Fix notation as shown in Figure \ref{fig:jump}. If $|x_i|\geq 6d$, the
    Hausdorff distance between $E_i$ and $E_{i+1}$ is bounded above by
    $2|l_ir_{i+1}|$.
\end{lemma}

\begin{figure}
    \centering
    \includegraphics[width=0.6\textwidth]{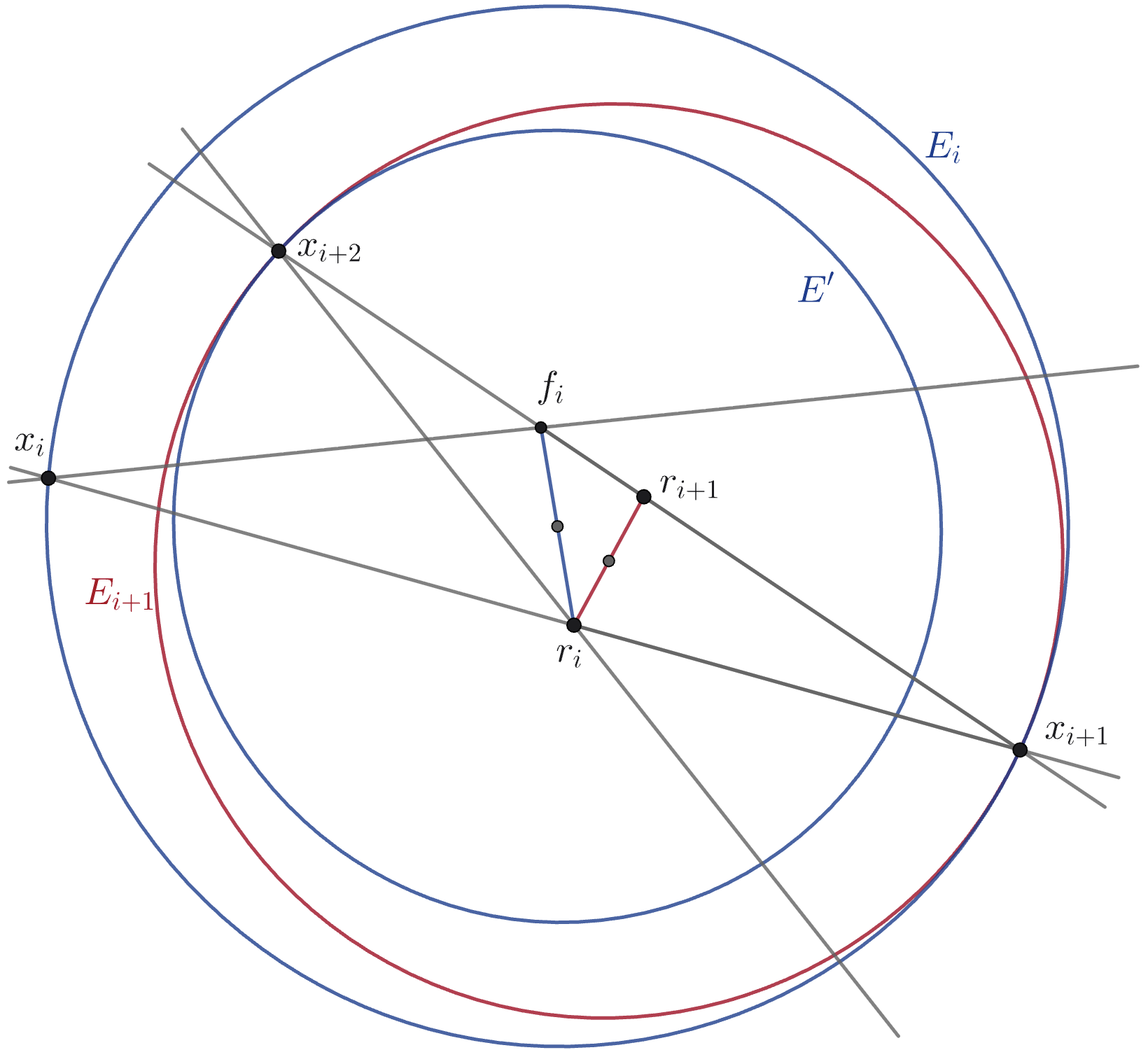}
    \caption{The configuration used in the proof of Lemma
    \ref{lem:unsteady}.}
    \label{fig:hausdorff_proof}
\end{figure}

\begin{proof}
    Notice that, by Lemma \ref{lem:obtuse}, $\angle x_if_ix_{i+1}$
    is obtuse and therefore 
    \begin{align}
        |f_ir_{i+1}|&\leq|l_ir_{i+1}|\\ \label{eq:focal_move}
                    &\leq d.
    \end{align}
    This, together
    with the inequality $|r_ir_{i+1}|\leq d$ and the triangle inequality imply
    $|r_if_i|\leq 2d$.

    By Lemma \ref{lem:foci} with $A=r_i$, $B=r_{i+1}$, $C=f_i$,
    and $D=x_{i+2}$, we must have that $E_i$ contains $E_{i+1}$.
    Applying the same lemma again, now with $A=r_i$, $B=r_{i+1}$, $C=f_i$,
    and $D=x_{i+2}$, we find that $E_{i+1}$ contains the ellipse $E'$ defined by
    $|f_iP|+|r_iP|=|f_ix_{i+2}|+|r_ix_{i+2}|$. Since $E_{i+1}$ is itself contained in $E_i$,
    we must have that
    \begin{equation}
        \label{eq:dh}
        d_H(E_i,E_{i+1})\leq d_H(E_i,E'),
    \end{equation}
    where $d_H$ is the
    Hausdorff distance.

    The intrafocal lengths of both ellipses are bounded above by $2d$, and
    by two applications of Lemma \ref{lem:uniform}, 
    $|r_ix_{i+2}|\geq 6d-\frac{3}{2}d-\frac{3}{2}d=3d$. By the triangle
    inequality, $|c_ix_i+2|\geq2d$, and we have that both $E_i$ and $E'$ have
    eccentricity at most $\frac{1}{2}$.

    Recall that $P$ lies on $E'$ exactly if
    $$|f_iP|+|r_iP|=|f_ix_{i+2}|+|r_ix_{i+2}|=:2a'$$
    while $E_i$ is defined by
    $$|f_iP|+|r_iP|=|f_ix_{i+1}|+|r_ix_{i+1}|=:2a.$$
    Using the collinearity of the points $x_{i+1}$, $r_{i+1}$, $f_i$, and
    $x_{i+2}$, we see that
    \begin{align*}
        2a&=|f_ix_{i+1}|+|r_ix_{i+1}|\\
          &=|f_ir_{i+1}|+|r_{i+1}x_{i+1}|+|r_ix_{i+1}|\\
          &=|f_ir_{i+1}|+|r_{i+1}x_{i+2}|+|r_ix_{i+2}|\\
          &=2|f_ir_{i+1}|+|f_ix_{i+2}|+|r_ix_{i+2}|\\
          &=2a'+2|f_ir_{i+1}|,
    \end{align*}
    where the third equality follows from the fact that $x_{i+2}$ lies on
    $E_{i+1}$. This computation shows $a-a'=|f_ir_{i+1}|$, so by Lemma
    \ref{lem:confocal}, we have $d_H(E_i,E')\leq 2|f_ir_{i+1}|$. After applying
    (\ref{eq:dh}) and then (\ref{eq:focal_move}), the proof is complete.
\end{proof}

\begin{proof}[Proof of Theorem \ref{thm:circle}]
    Let $d$ be the diameter of $Q$.
    Under iteration by the square of the outer length billiard
    map, the orbit progresses along steady phases punctuated by unsteady points.
    We bound the maximum possible radial deviation by adding together two types
    of deviation: that from the eccentricity of the ellipses traversed during
    the steady phases (which we term ``eccentricity error''), and that which
    accrues during the jumps from ellipse to ellipse at the unsteady points
    (which we call ``jump error'').

    Suppose that $|x_i|\geq 8d$ for all $x_i\in\OO(x)$, so that by Lemma
    \ref{lem:obtuse}, all virtual tables have length at most $2d$.
    Since $Q$ contains the origin, Lemma \ref{lem:eccentricity} applies to the
    first ellipse.
     
    Lemma \ref{lem:unsteady} relates each jump to a side length or diagonal
    of $Q$. The once-around orbit sees the points $l_i$ move around $Q$, but not
    more than twice, so we can add up all jump errors to find that the final
    ellipse has Hausdorff distance not more than 
    $$
    \sum_{i~\text{ unst.}}d_H(E_i,E_{i+1})
    \leq2\sum_{i~\text{unst.}}|r_{i+1}l_i|\leq 4p$$
    from the first, where $p$ is the perimeter of $Q$.

    Therefore, all points in $\OO$ lie within $4p$ units of an ellipse which
    satisfies the assumptions of Lemma \ref{lem:eccentricity}. This gives that
    $$C_1=4p+3d\leq (4\pi+3)d.$$
    The assumption that all points of $\OO$ are at least $8d$ from the origin
    will thus be met if the starting point $x$ satisfies $|x|\geq C_2 :=
    (4\pi+11)d$.
\end{proof}

\begin{remark}
    This argument can with high likelihood be extended to cover all convex
    bodies. We hope to prove this fact in a future work.
\end{remark}

\section{Circle centers}
\label{sec:centers}

\begin{figure}
    \centering
    \includegraphics[width=0.5\textwidth]{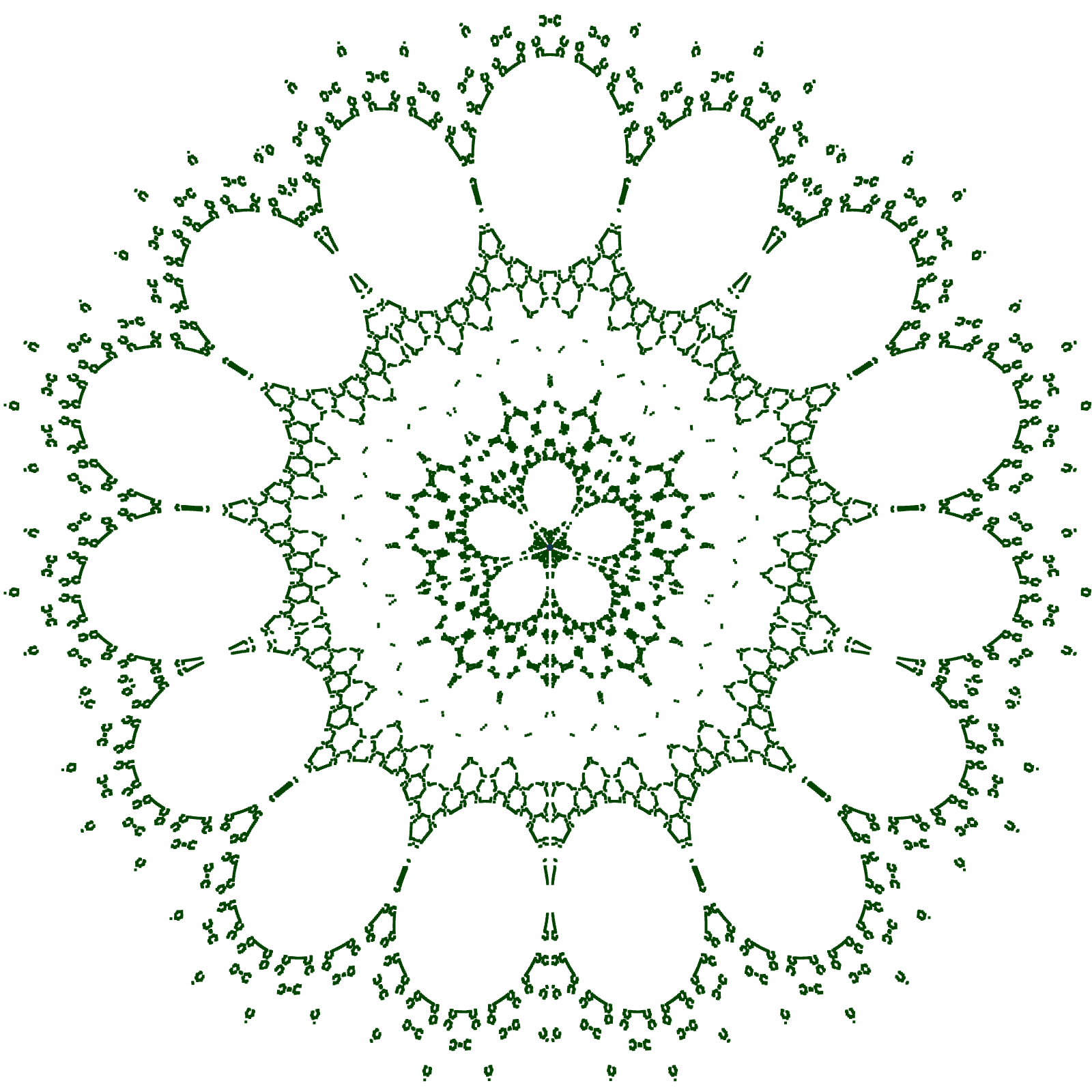}
    \caption{A single orbit of the auxiliary circle centers around a regular
    pentagon, which is too small to see at the center of the image.}
    \label{fig:cloud}
\end{figure}

The centers of the auxiliary circles are dynamically interesting in their own
right. We can ``reverse engineer'' the construction of the outer length billiard
map from the center of an auxiliary circle as follows.

Let $C$ be a point outside $K$ and let $P$ be the closest point of $\partial K$
to $C$. Suppose further that $P$ does not lie on a line segment contained in
$\partial K$. Then the circle centered at $C$ and passing through $P$ has three
common support lines with $K$, which generically have three pairwise points of
coincidence. The two of these which lie on the line passing through $P$ are
related to one another via the outer length billiard map.

This construction shows that to each trajectory of the outer length billiard map
is associated a trajectory of the centers of the auxiliary centers. The centers
corresponding to a complicated orbit is shown in Figure \ref{fig:cloud}
\begin{proposition}
    \label{prop:mid}
    Let $Q$ be a convex $n$-gon, $x$ be a point at which both forward and
    reverse outer length billiard maps are defined. If $C_+$ and $C_-$ are the
    centers of the auxiliary circles used to compute the forward and reverse
    maps, then the midpoint $M$ between $C_+$ and $C_-$ lies on the
    perpendicular bisector of one of the sides or diagonals of $Q$.
\end{proposition}

\begin{proof}
    Let $v_+$ and $v_-$ be the vertices of $K$ closest to $C_+$ and $C_-$,
    respectively; note that these vertices may not be adjacent. Denote by
    $r_\pm$ the radii of the two circles.
    Let $\ell$ be the line containing $v_+$ and $v_-$. Let $F_\pm$ be the
    orthogonal projections of $C_\pm$ to $\ell$. It suffices to show that
    $|v_+F_+|=|v_-F_-|$.

    \begin{figure}
        \centering
        \includegraphics[width=0.8\textwidth]{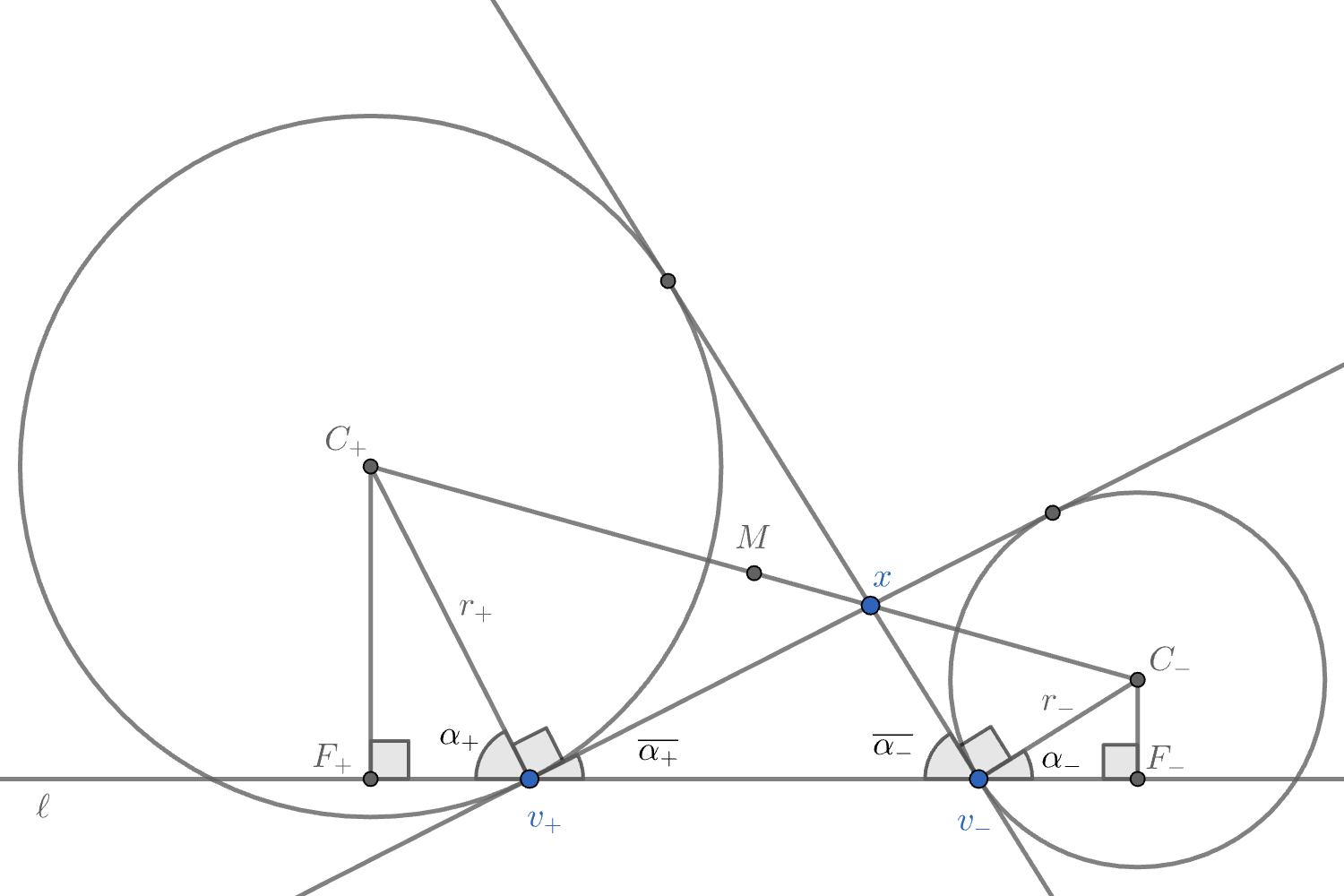}
        \caption{The forward and reverse auxiliary circle centers for point $x$.
        The table $Q$ is not shown, but it has vertices $v_+$ and $v_-$ and does
        not otherwise intersect the lines $\overline{xv_+}$ and
        $\overline{xv_-}$.}
        \label{fig:circles}
    \end{figure}

    With angles as labeled in Figure \ref{fig:circles}, we have
    \begin{align*}
        |v_+F_+|&=r_+\cos(\alpha_+)=r_+\sin(\overline{\alpha_+})\\
        |v_-F_-|&=r_-\cos(\alpha_-)=r_-\sin(\overline{\alpha_-}).
    \end{align*}
    By the Law of Sines,
    $$\frac{\sin(\overline{\alpha_+})}{\sin(\overline{\alpha_-})}
    =\frac{|xv_-|}{|xv_+|}.$$
    $\triangle xv_+C_+\sim\triangle xv_-C_-$, so
    $$\frac{|xv_-|}{|xv_+|}=\frac{r_-}{r_+}.$$
    Combining, we have
    \begin{align*}
        |v_+F_+|&=r_+\sin(\overline{\alpha_+})\\
                &=r_+\left(\sin(\overline{\alpha_-})\frac{|xv_-|}{|xv_+|}\right)\\
                &=r_+\left(\sin(\overline{\alpha_-})\frac{r_-}{r_+}\right)\\
                &=r_-\sin(\overline{\alpha_-})\\
                &=|v_-F_-|,
    \end{align*}
    so, indeed, the point $M$ lies on the perpendicular bisector of the segment
    $v_+v_-$. This segment is either a side or a diagonal of $Q$, so we have
    proved the proposition.
\end{proof}

Proposition \ref{prop:mid} suggests that midpoints between consecutive circle
centers are of interest and hint at a connection with the usual outer billiards.
This connection exists and is described in Theorem \ref{thm:centers} and
exemplified in Figure \ref{fig:hexagons}.
Let us fix some notation: let $\chi$ be the map which takes a point $x$ to the
center of the auxiliary circle which is constructed when finding $T(x)$. 
\begin{figure}
    \centering
    \includegraphics[width=0.48\textwidth]{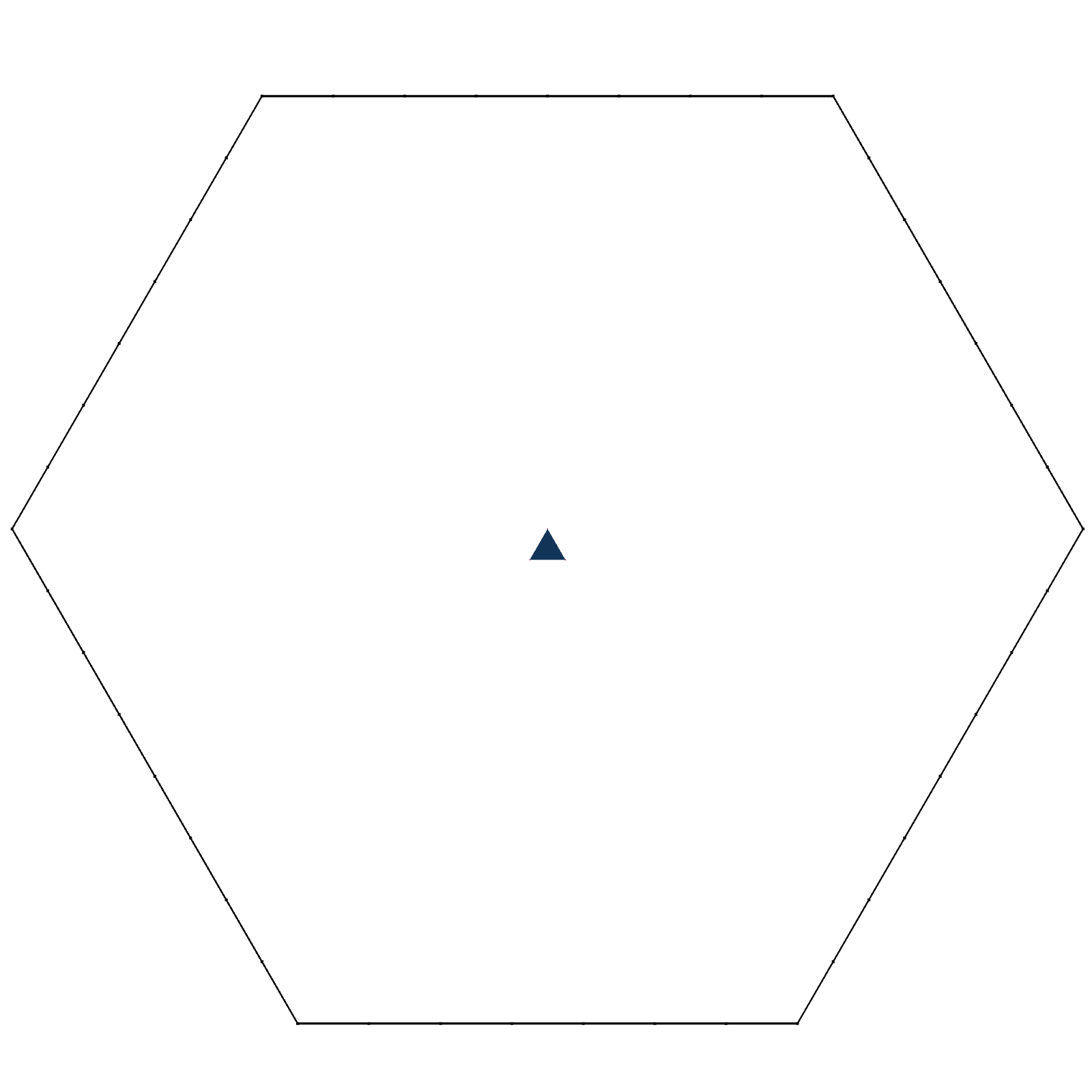}
    \hfill
    \includegraphics[width=0.48\textwidth]{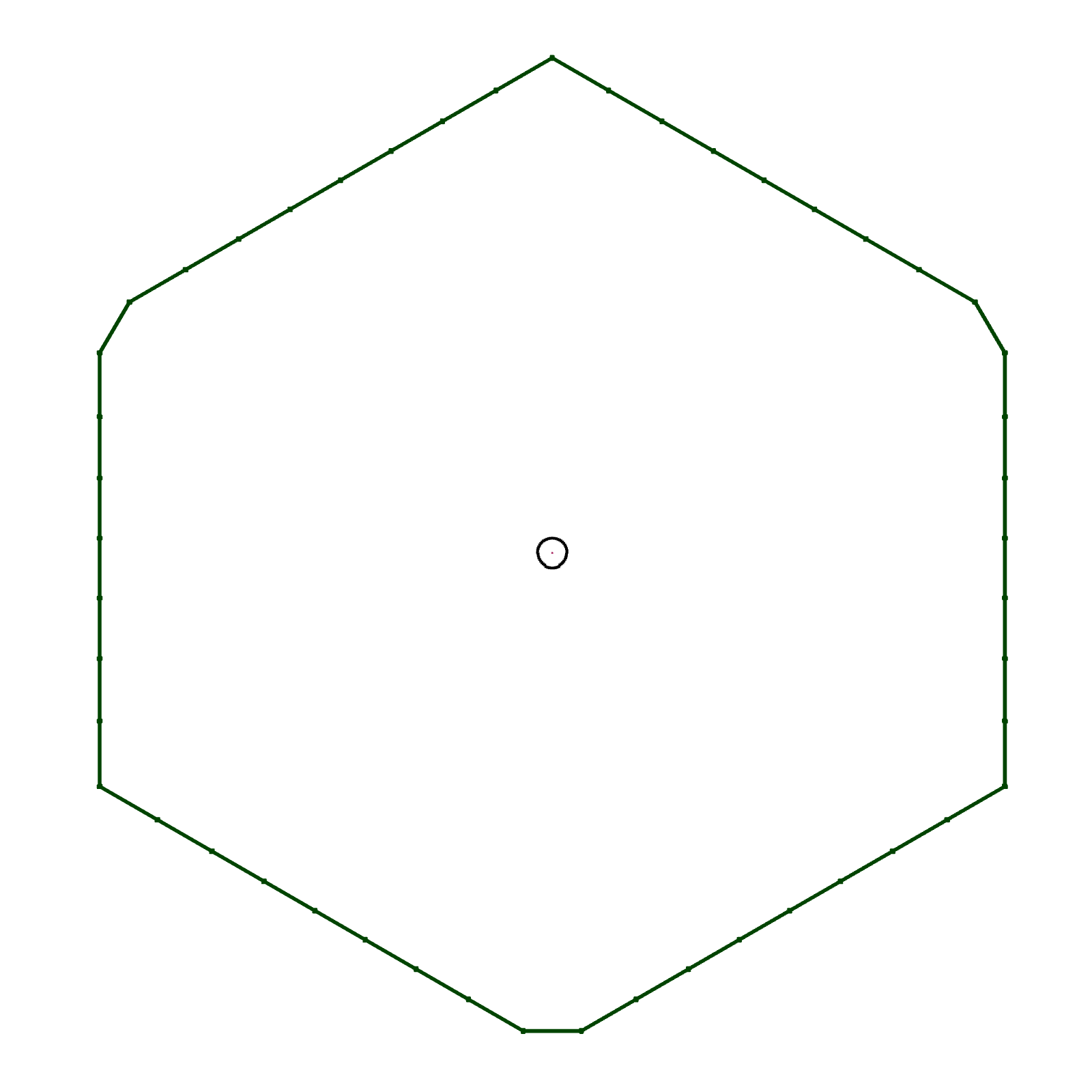}
    \caption{Left: a distant orbit of the usual outer billiard about an
        equilateral triangle. Right: a distant orbit of the outer length
        billiard about the same triangle is shown as a small black circle while
        the centers are shown in dark green. In all cases, each point is
    connected to its second image by a line segment for clarity.}
    \label{fig:hexagons}
\end{figure}

\begin{theorem}
    \label{thm:centers}
    Let $K$ be a convex polygon of diameter 1 whose interior contains the
    origin and let $C_1$ and $C_2$ be the constants from Theorem
    \ref{thm:circle}.
    Let $Q=d K$ where $0<d<1/C_2$. Let $\OO$ be the once-around
    orbit of a point $x$ on the unit circle; observe that $\OO$ is constrained
    to annulus of width $2dC_2$. 
    Further, let $\Gamma(\theta)$ be the radial function of the symplectic polar
    dual to the symmetrized
    version of $Q$. Then there exist positive constants $\epsilon$, and $C_3$
    depending on $Q$ such that if $d\leq\epsilon$ and $x=(r,\theta)$ is a point
    in $\OO$, then
    $$\norm{\chi(x)-\Gamma(\theta)\langle-\sin(\theta),\cos(\theta)\rangle} \leq C_3.$$
    In words, the once around orbit $\CC$ of the circle center map is
    constrained to a neighborhood of a rotated copy of a distant orbit of the
    outer area billiard about $Q$.
\end{theorem}

In the discussion that follows, let $w(\theta)$ be the width of the table as
viewed from direction $\theta$, and let $k$ be the ``aspect ratio'' of $Q$:
$$k=\frac{d}{\min_\theta w(\theta)}\in[1,\infty).$$

\begin{lemma}
    \label{lem:phi}
    If $x=(r,\theta)$ in polar coordinates, with $|r-1|\leq C_1$ (the width of
    the annulus from Theorem \ref{thm:circle}), then the angle
    $\phi=\angle pxq$ (where $p=L(x)$ and $q=R(x)$ as in Figure \ref{fig:phi})
    satisfies
    $$\phi=w(\theta)+O(d^2).$$
\end{lemma}

\begin{figure}
    \centering
    \includegraphics[width=0.5\textwidth]{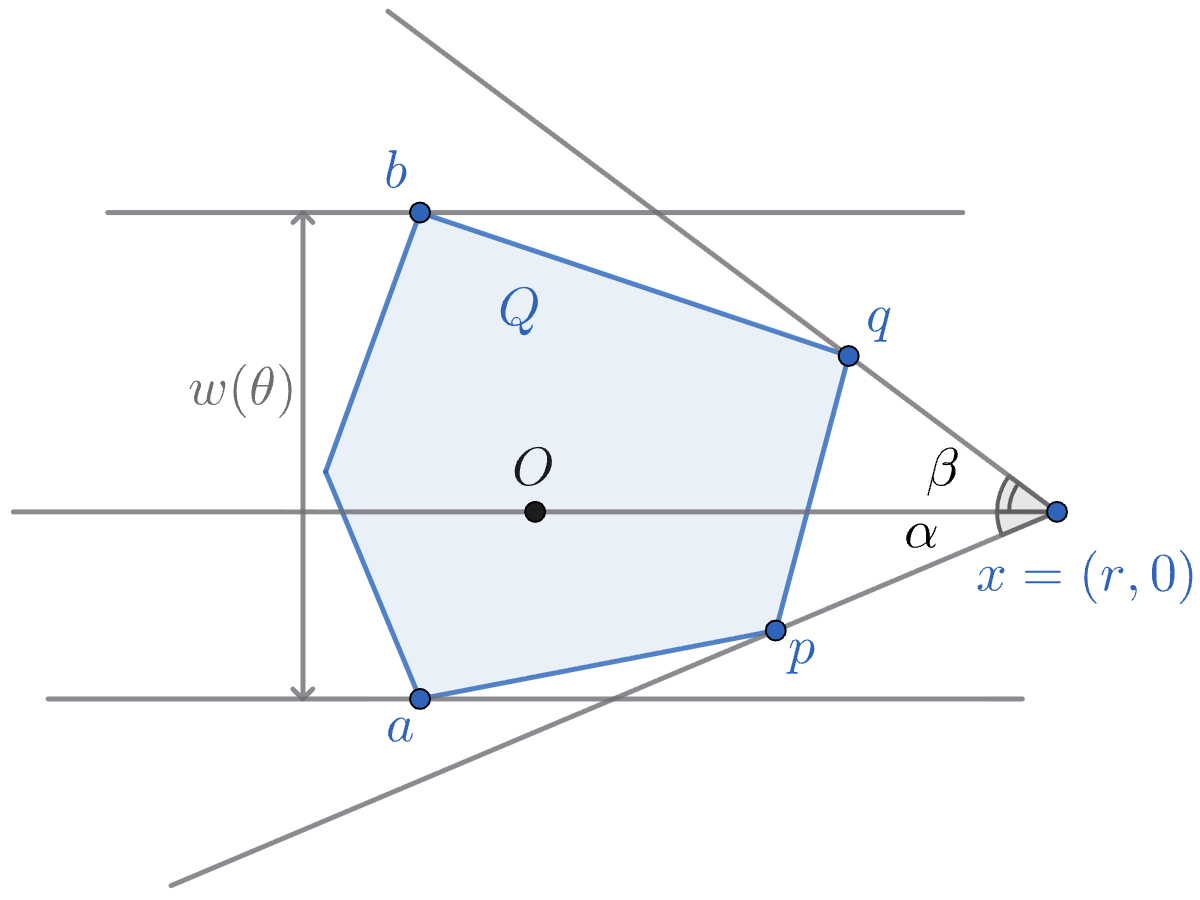}
    \caption{The configuration considered in the proof of Lemma \ref{lem:phi}.}
    \label{fig:phi}
\end{figure}

\begin{proof}
    Rotate the coordinate system so that $x=(r,0)$. With notation as in Figure
    \ref{fig:phi}, we have $\phi=\alpha+\beta$. Points $a$ and $b$ are
    extreme points of $Q$ such that $w(\theta)=w(0)=|a_y|+|b_y|$.

    We will bound $\alpha$, first from below. Let $m_a=\frac{a_y}{a_x-r}$ be the
    slope of the line passing through $a$ and $(1,0)$. Notice that, since
    $|a_x|,|a_y|\leq d$,
    $$m_a=|a_y|\left(\frac{1}{r-a_x}\right)
    =|a_y|\cdot\frac{1}{r}\left(1+\frac{a_x}{r}+\frac{a_x^2}{r^2}+\cdots\right)$$
    We can write $r=1+O(d)$ and thus
    $$\frac{1}{r}=(1+O(d)+O(d)^2+\cdots)=1+O(d)$$
    so we have
    $$\frac{a_x}{r}=a_x+a_xO(d)=a_x+O(d^2).$$
    Together, this gives
    \begin{align*}
        m_a&=|a_y|(1+O(d))\left(1+(a_x+O(d^2))+(a_x+O(d^2))^2+\cdots\right)\\
           &=|a_y|+|a_y|a_x+|a_y|O(d)+O(d^3)\\
           &=|a_y|+O(d^2).
    \end{align*}

    It must be that
    $\alpha\geq\angle Oxa=\arctan(m_a)$, so we Taylor expand to find that
    \begin{align*}
        \alpha&=\arctan(m_a)\\
              &=m_a - \frac{m_a^3}{3}+\frac{m_a^5}{5}+\dots\\
              &=|a_y|+O(d^2)+O(d^3)\\
              &=|a_y|+O(d^2),
    \end{align*}
    where we make use of the fact that, of course, $m_a=O(d)$.

    Now, we show an upper bound for $\alpha$. Let $m_p=\frac{p_y}{p_x-r}$ and
    notice that $p_y\geq a_y$.
    We have that
    \begin{align*}
        \alpha&=\arctan(m_p)\\
              &\leq m_p\\
              &=\frac{-p_y}{1-p_x}\\
              &\leq |a_y|\left(\frac{1}{r-p_x}\right)\\
              &=|a_y|\cdot \frac{1}{r}\left(1+\frac{p_x}{r}+\frac{p_x^2}{r^2}+\cdots\right)\\
              &=|a_y|+O(d^2),
    \end{align*}
    where the last two inequalities follow similar logic to that used in the lower
    bound.
    
    We now have both upper and lower bounds for $\alpha$ of the form
    $|a_y|+O(\epsilon^2)$, so $\alpha$ itself must have the same form. An
    identical argument for $\beta$ allows to write that
    \begin{align*}
        \phi&=\alpha+\beta\\
            &=(|a_y|+O(d^2))+(|b_y|+O(d^2))\\
            &=|a_y|+|b_y|+O(d^2)\\
            &=w(0)+O(d^2)
    \end{align*}
    as desired.
\end{proof}

\begin{lemma}
    \label{lem:radius}
    Let $x=(r,\theta)$ with $|r-1|\leq C_1$. Then the radius $\rho$ of the
    auxiliary circle centered at $\chi(x)$ satisfies
    $$\rho=\frac{2}{w(\theta)}+O(1).$$

\end{lemma}

\begin{figure}
    \centering
    \includegraphics[width=0.5\textwidth]{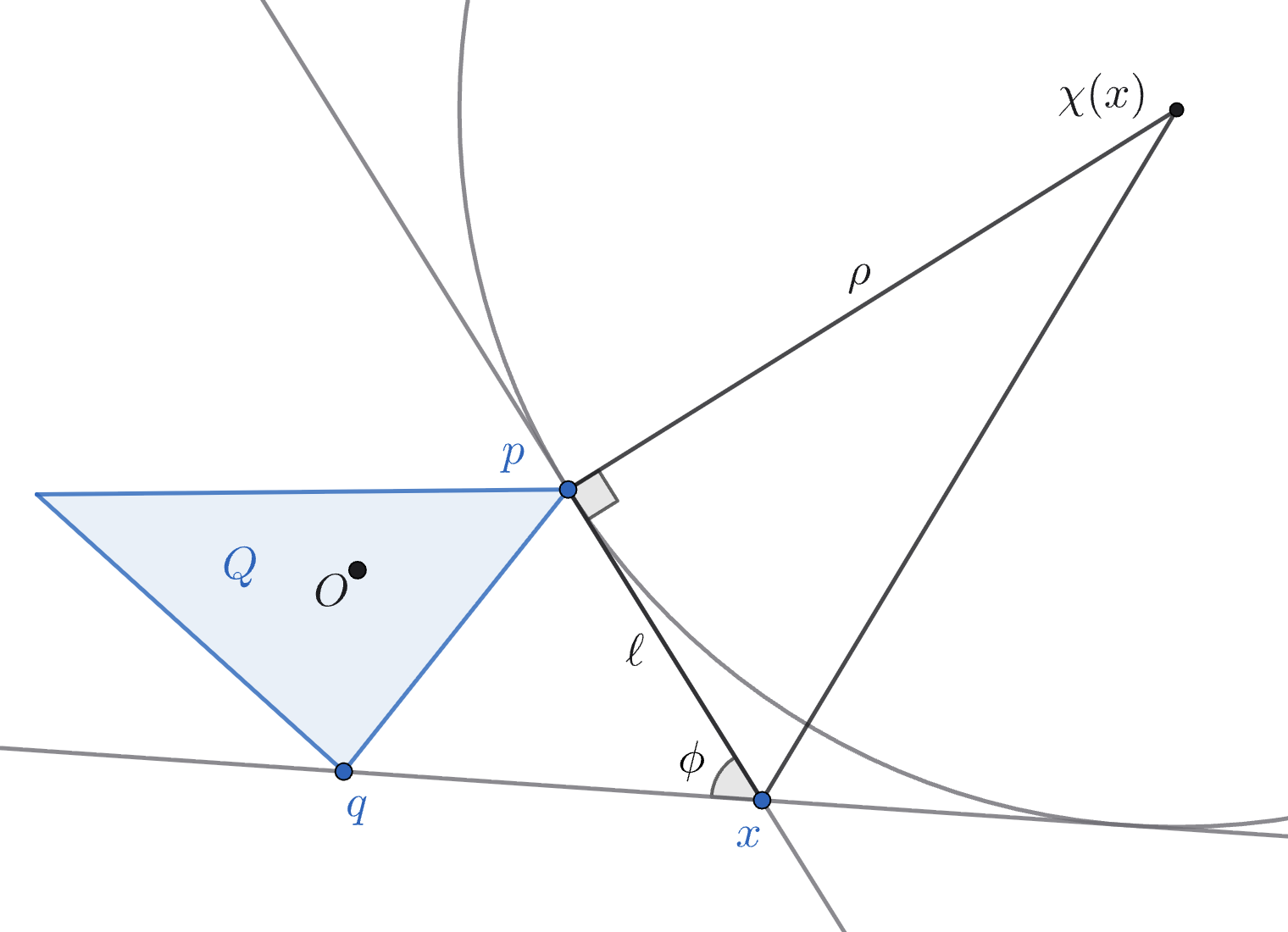}
    \caption{The configuration considered in the proof of Lemma
    \ref{lem:radius}.}
    \label{fig:radius}
\end{figure}

\begin{proof}
    Fix notation as in Figure \ref{fig:radius}. By the triangle inequality, we
    have that $|\ell-r|\leq d$, which gives $|\ell - 1|\leq C_1+d$.
    Some ``angle chasing'' reveals that $\angle p\chi(x) x=\phi/2$.
    Thus, $\rho=\ell\cot(\phi/2)$. Using the Taylor expansion of the cotangent
    function and Lemma \ref{lem:phi} with the observation that
    $\phi = w(\theta) + O(d^2)=O(d)$, we have
    \begin{align*}
        \rho&=\ell\left(\frac{2}{\phi}-\frac{\phi}{6}+O(\phi^3)\right)\\
         &=\frac{2\ell}{\phi}-\frac{(1+O(d))\phi}{6}+(1+O(d))O(d^3)\\
         &=\frac{2\ell}{\phi}+O(d).
    \end{align*}
    We must now be slightly more precise in handling the remaining term. First,
    notice that
    \begin{align*}
        \frac{2}{\phi}&=\frac{2}{w(\theta)+O(d^2)}\\
                      &=\frac{2}{w(\theta)}
                      \left(1+\frac{O(d^2)}{w(\theta)}+\frac{O(d^2)^2}{w(\theta)^2}+\cdots\right)
    \end{align*}
    Since $w(\theta)$ is bounded below by $d/k$, we can safely replace
    $O(d^2)/w(\theta)$ with $O(d)$ and $O(d)/w(\theta)$ with $O(1)$, yielding
    $$\frac{2}{\phi}=\frac{2}{w(\theta)}(1+O(d))=\frac{2}{w(\theta)}+O(1)$$
    We obtain
    \begin{align*}
        \rho&=\frac{2}{\phi\ell}+O(d)\\
         &=\left(\frac{2}{w(\theta)}+O(1)\right)(1+O(d))+O(d)\\
         &=\frac{2}{w(\theta)}+\frac{2O(d)}{w(\theta)}+O(1)\\
         &=\frac{2}{w(\theta)}+O(1).
    \end{align*}
\end{proof}

\begin{proof}[Proof of Theorem \ref{thm:centers}]
    First, we must make precise what the curve $\Gamma$ is. Referring to Section
    \ref{sec:inf_area} and Figure \ref{fig:qsstar} we see that $Q_s$ has support
    function $h_{Q_s}(\theta)=w(\theta+\pi/2)/2$, where, again,
    $$w(\theta)=h_Q(\theta+\pi/2)+h_Q(\theta-\pi/2)$$
    gives the width of $Q$ as viewed from angle $\theta$. Therefore, $(Q_s)^*$
    has radial function
    $$\Gamma(\theta)=\frac{1}{h_{Q_s}(\theta+\pi/2)}=\frac{2}{w(\theta)}.$$
 
    \begin{figure}
        \centering
        \includegraphics[width=0.8\textwidth]{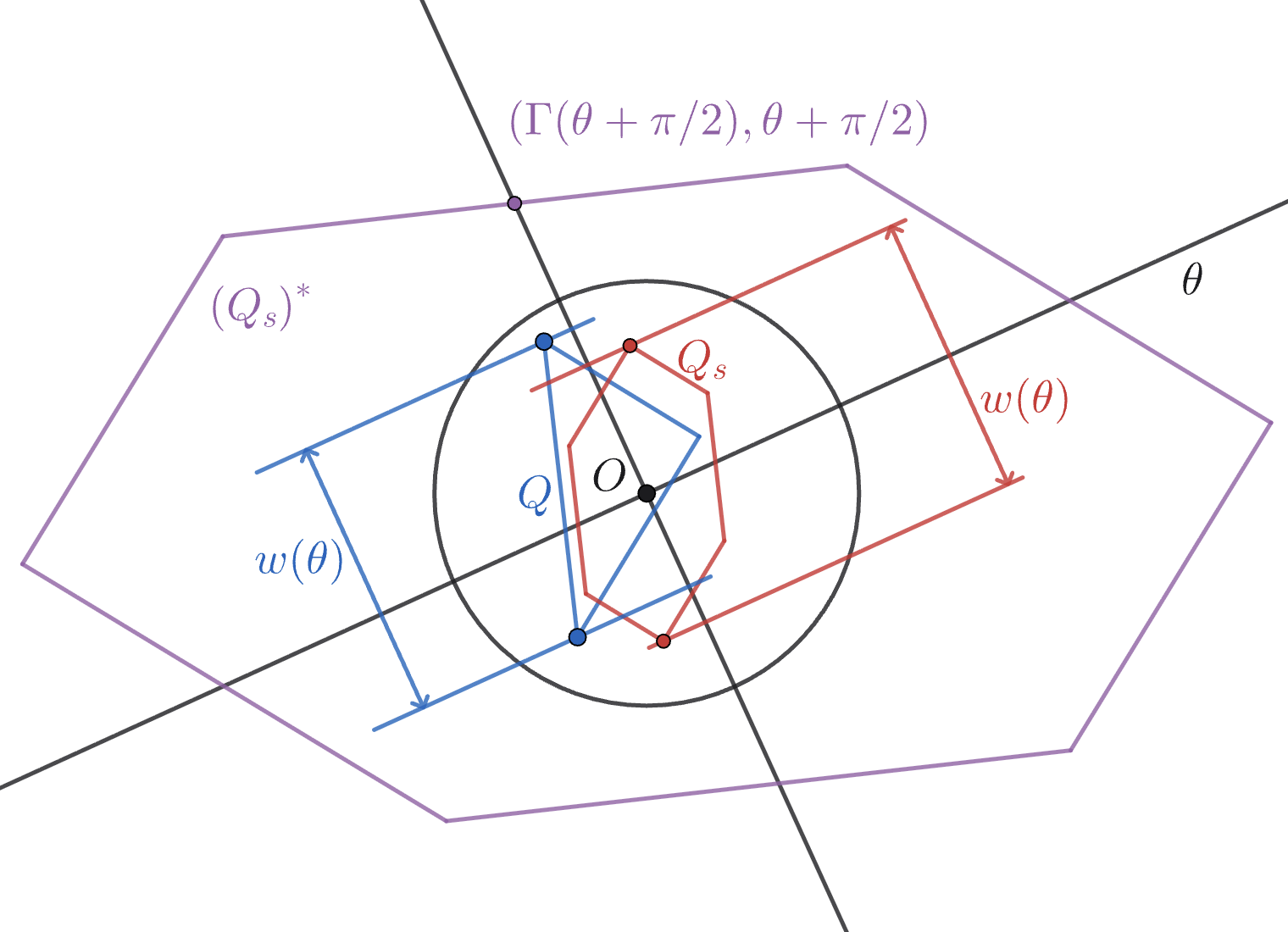}
        \caption{A polygon $Q$ in blue, its symmetrization $Q_s$ in red, and the
            polar dual $Q_s^\dagger$ in purple.}
        \label{fig:qsstar}
    \end{figure}

    Referring to Figure \ref{fig:radius} and Lemma \ref{lem:radius}, we notice
    that the hypotenuse of
    $\triangle p x \chi(x)$ has length $\rho+O(\ell)=\rho+O(1)$ by the triangle
    inequality.
    See also that $\psi:=\angle O x \chi(x)$ satisfies
    $$\frac{\pi}{2}-\frac{\phi}{2}\leq\psi\leq\frac{\pi}{2}+\frac{\phi}{2}.$$
    Put another way, $|\psi-\pi/2|\leq\frac{\phi}{2}=w(\theta)/2+O(d^2)$.

    Thus $\chi(x)-\vec{x}$ can be written in polar coordinates as
    $$\left(\rho+O(1),\theta+\frac{\pi}{2}+O(\phi)\right)
    =\left(\frac{2}{w(\theta)}+O(1),\theta+\frac{\pi}{2}+O(d)\right)$$
    Using once again that $w(\theta)$ is bounded below, we see that the $O(d)$
    error term in the angular coordinate only results in $O(1)$ distance in
    space. Finally, the offset $\vec{x}$ is absorbed into the $O(1)$ term,
    leaving
    $$\chi(x)=\frac{2}{w(\theta)}\langle -\sin(\theta),\cos(\theta) \rangle + O(1),$$
    as desired. Interpreting the big-$O$ notation as a pair of constants
    completes the proof.

\end{proof}

\section{Experimental results and conjectures}
\label{sec:experiment}

\subsection{Orbit and singularity structure}
\label{sec:periodic}

After much computer experimentation, we are led to believe that periodic,
quasi-periodic, non-periodic, and escaping orbits are all plentiful in polygonal
outer length billiards.

\begin{conjecture}
    Every polygonal outer length billiard admits a periodic orbit of length 3.
\end{conjecture}

This is resolved for triangular tables in Section \ref{sec:triangle}. An example
around an irregular pentagon is shown in Figure \ref{fig:irregular}.

\begin{figure}
    \centering
    \includegraphics[width=0.5\textwidth]{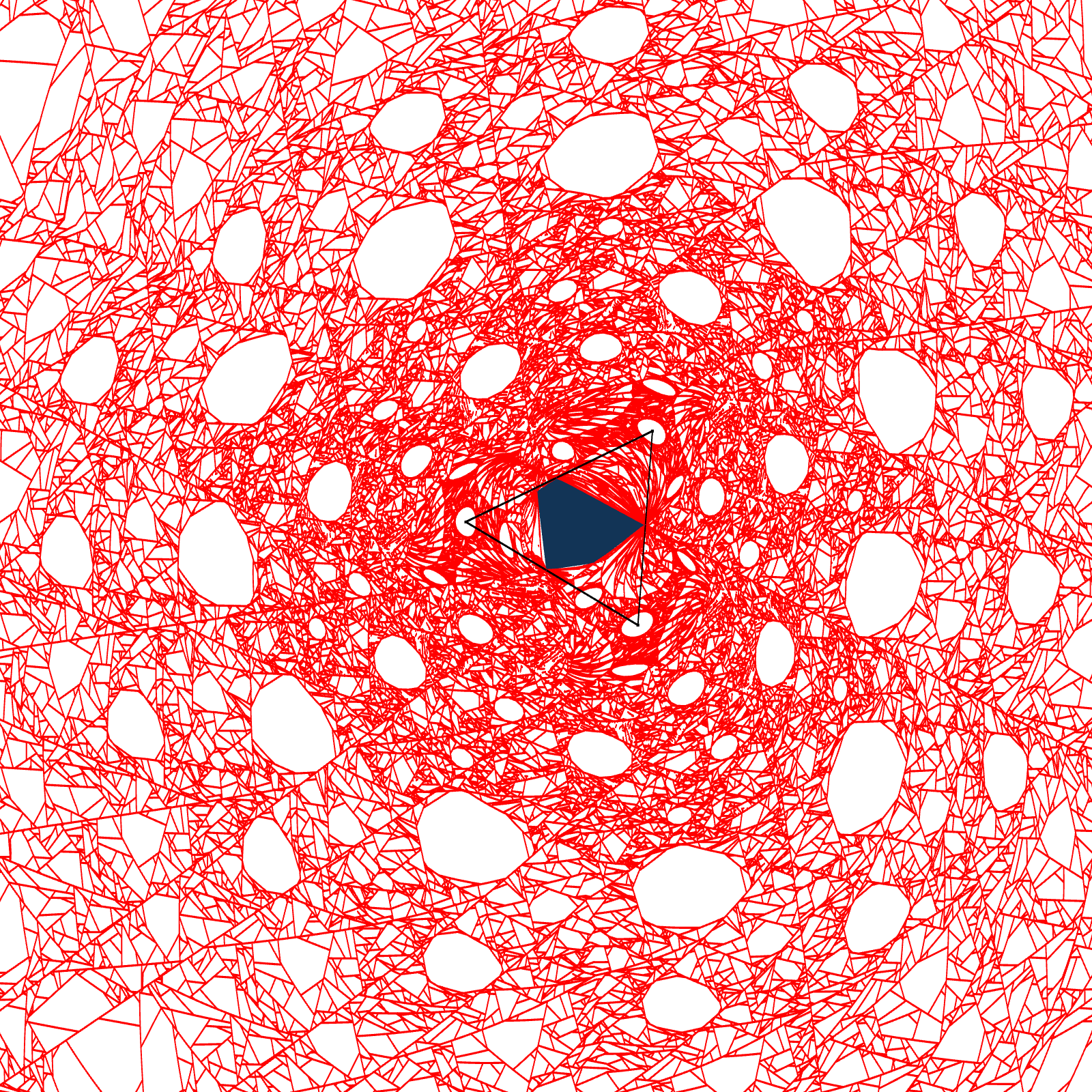}
    \caption{A coarse approximation of the singularity set of an irregular
        pentagon together with a 3-periodic orbit.}
    \label{fig:irregular}
\end{figure}

The question of existence of escaping orbits for polygonal outer billiards was
open for decades before being resolved in the positive by Schwartz, who showed
in 2009 that an irrational kite gives rise to orbits which diverge to infinity,
albeit in a highly non-monotonic manner \cite{kite1,kite2}.

Here, we address the same question for the outer length billiard.

\begin{question}
    Does there exist a polygon $Q$ for which the outer length billiard has
    unbounded orbits?
\end{question}

A number of computer experiments support the following conjectures.

\begin{conjecture}
    The singularity set of every polygon has Hausdorff dimension greater than 1.
    Furthermore, for every polygon, there exists an orbit whose closure is the
    singularity set, and which is therefore unbounded.
\end{conjecture}

\begin{figure}
    \centering
    \includegraphics[width=0.48\textwidth]{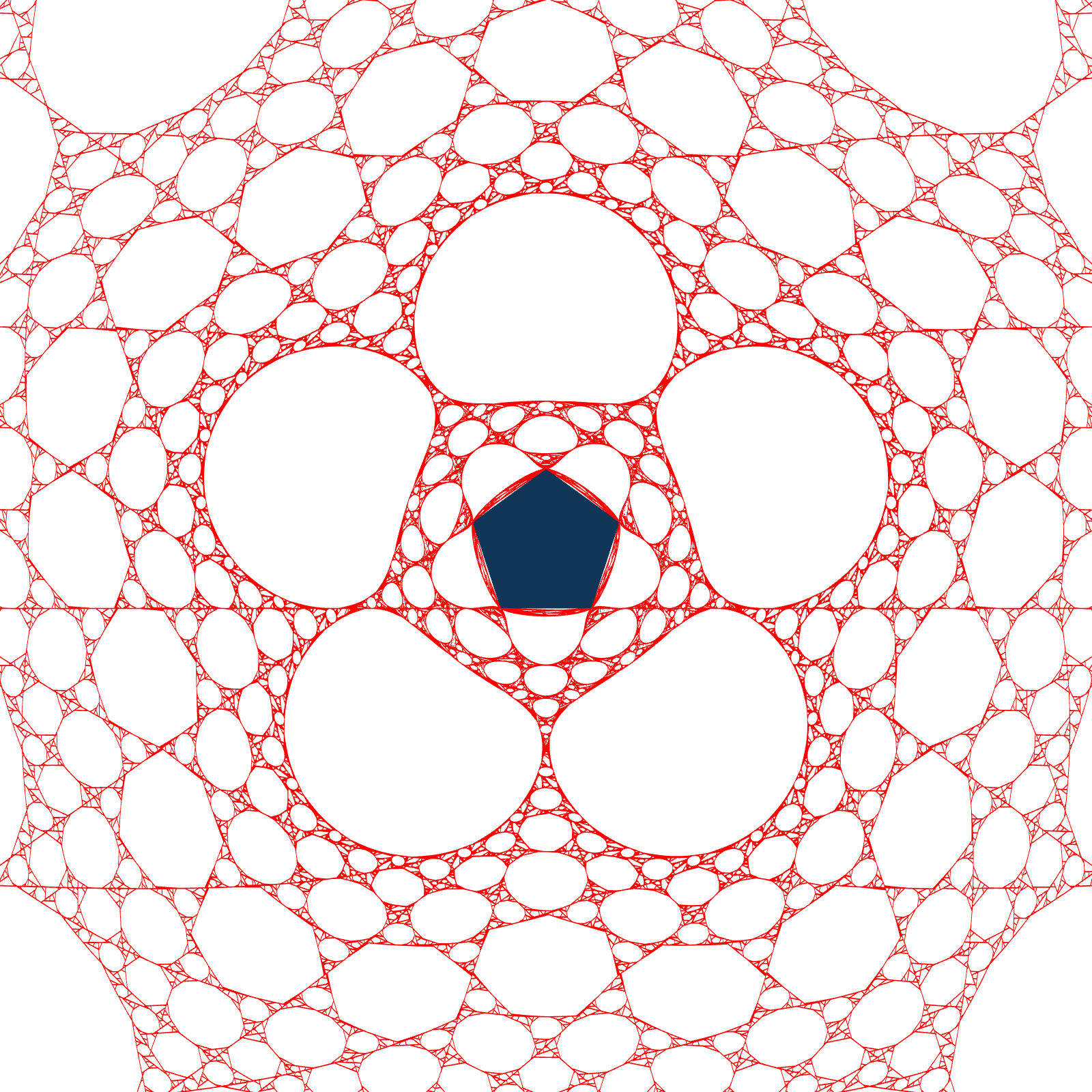}
    \hfill
    \includegraphics[width=0.48\textwidth]{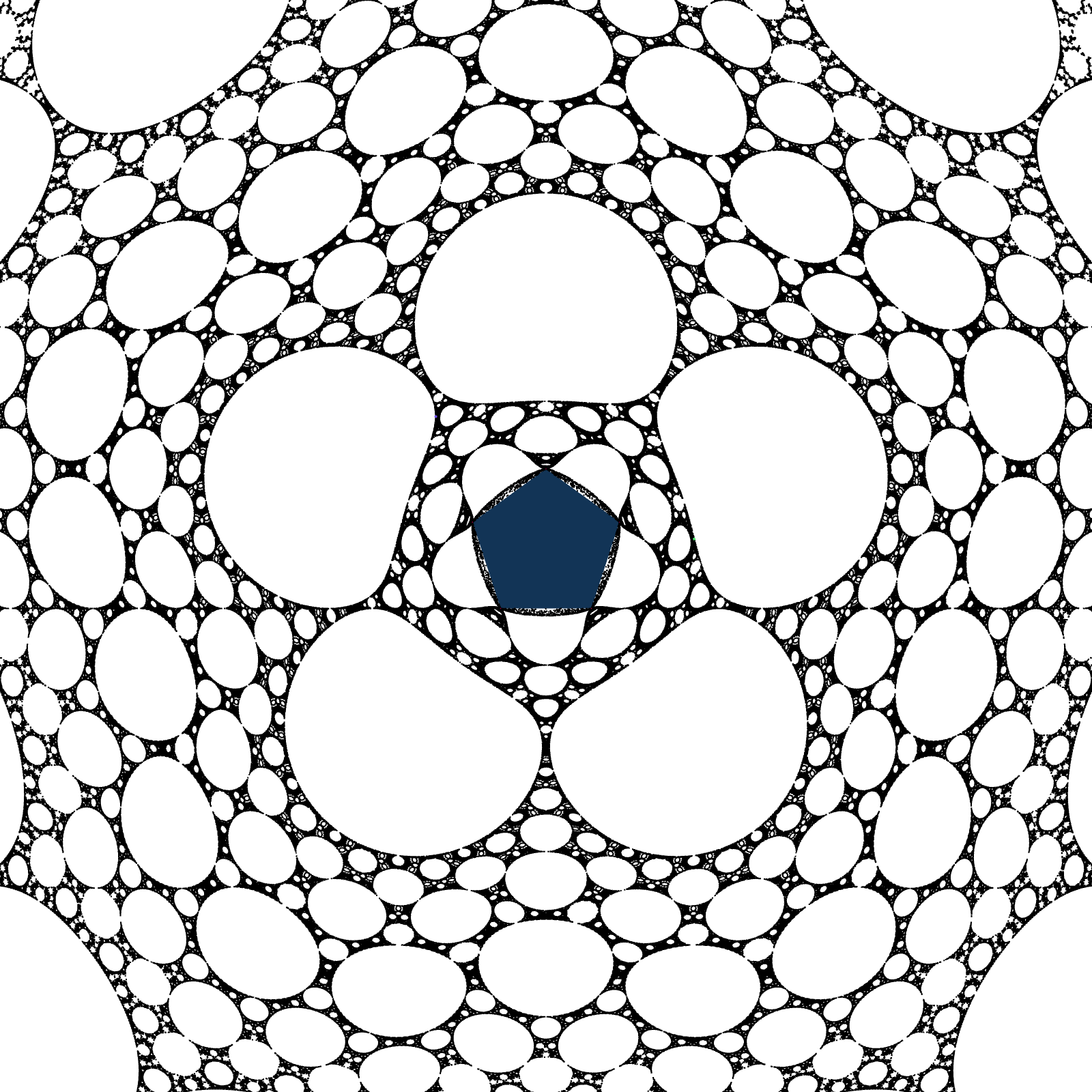}
    \caption{Left: the singularity set of the regular pentagon. Right: a single
    orbit which is approximately dense in the singularity.}
    \label{fig:dense}
\end{figure}

A sample orbit around a regular pentagon is shown in Figure \ref{fig:dense},
notice how the singularity set and orbit are visually identical.

\subsection{The square}
\label{sec:square}

\begin{conjecture}
    \label{conj:square}
    There exists an open ball of diverging orbits of the outer length billiard
    around a square.
\end{conjecture}

Conjecture \ref{conj:square} is likely more tractable to prove, so we will go
into some detail about what appears to be true.

Let $Q$ be the square with vertices at $(\pm1,\pm1)$; let $v_i$ be the vertex in
the $i$-th quadrant, $i=1,2,3,4$. A given point $x$ at which the outer length
billiard map is defined gives rise to three lines, $\ell_1$, $\ell_2$, and
$\ell_3$ (with notation as in the construction in Section \ref{s:construction}).
Each of these lines passes through a single vertex of $Q$. We therefore denote
by $T_{ijk}$ the piece of the map such that $\ell_1$ passes through $v_i$, etc.
One can easily verify that the combinations which occur for the square are:
$$T_{123},\quad T_{124},\quad T_{131},\quad T_{134},$$
together with their cyclic permutations (mod 4). This yields 16 pieces of the
map, 12 of which have unbounded domains, as shown in Figure \ref{fig:square}.

\begin{figure}
    \centering
    \begin{tikzpicture}
    \node[anchor=south west, inner sep=0] (img) at (0,0)
        {\includegraphics[width=0.48\textwidth]{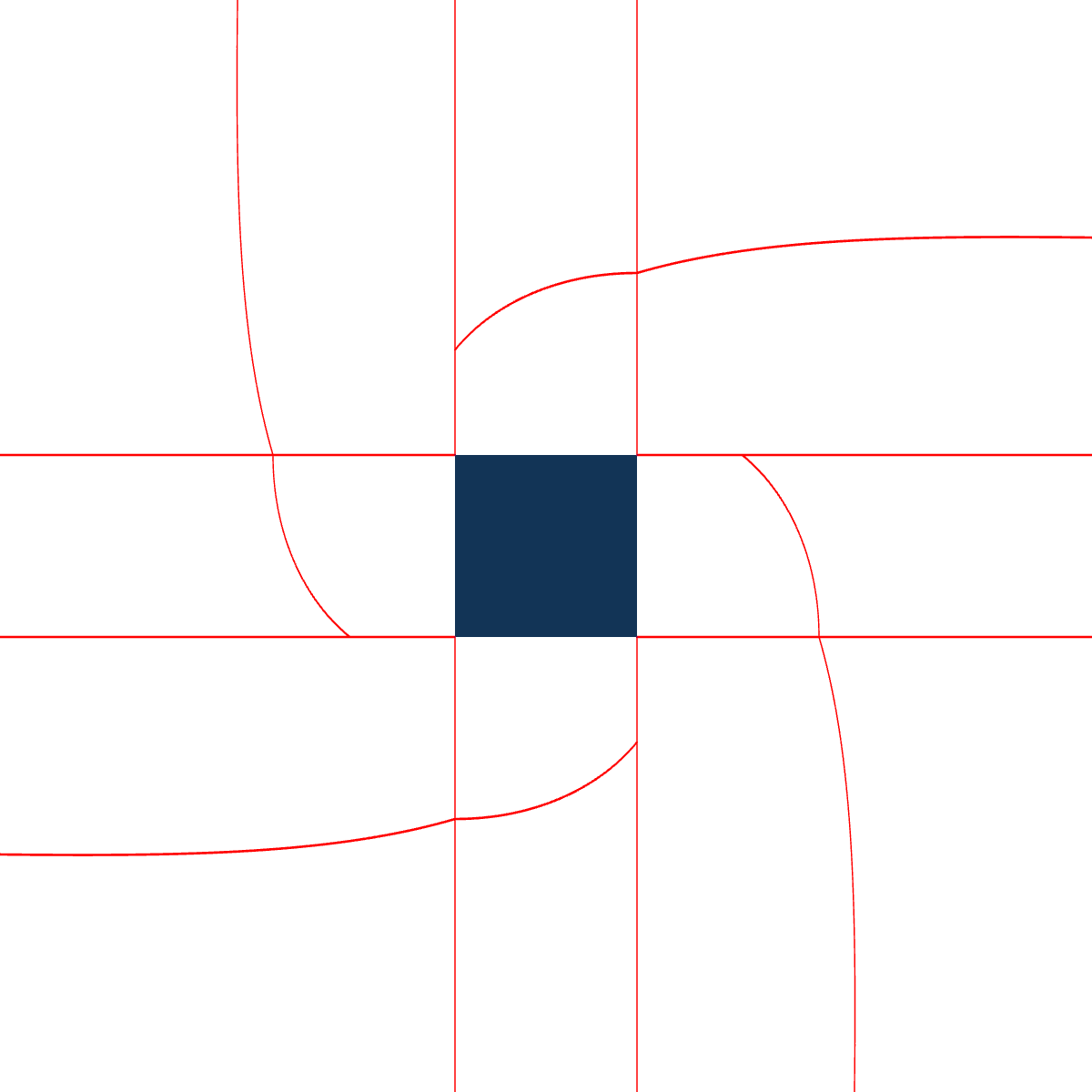}};
    \begin{scope}[x={(img.south east)}, y={(img.north west)}]
        \node at (0.5,0.65) {$T_{123}$};
        \node at (0.5,0.9) {$T_{124}$};
        \node at (0.1,0.7) {$T_{131}$};
        \node at (0.3,0.8) {$T_{134}$};
        
        \node at (0.35,0.5) {$T_{234}$};
        \node at (0.1,0.5) {$T_{231}$};
        \node at (0.3,0.1) {$T_{242}$};
        \node at (0.2,0.3) {$T_{241}$};

        \node at (0.5,0.35) {$T_{341}$};
        \node at (0.5,0.1) {$T_{342}$};
        \node at (0.9,0.3) {$T_{313}$};
        \node at (0.7,0.2) {$T_{312}$};

        \node at (0.65,0.5) {$T_{412}$};
        \node at (0.9,0.5) {$T_{413}$};
        \node at (0.7,0.9) {$T_{424}$};
        \node at (0.8,0.7) {$T_{423}$};
    \end{scope}
  \end{tikzpicture}
    \hfill
    \includegraphics[width=0.48\textwidth]{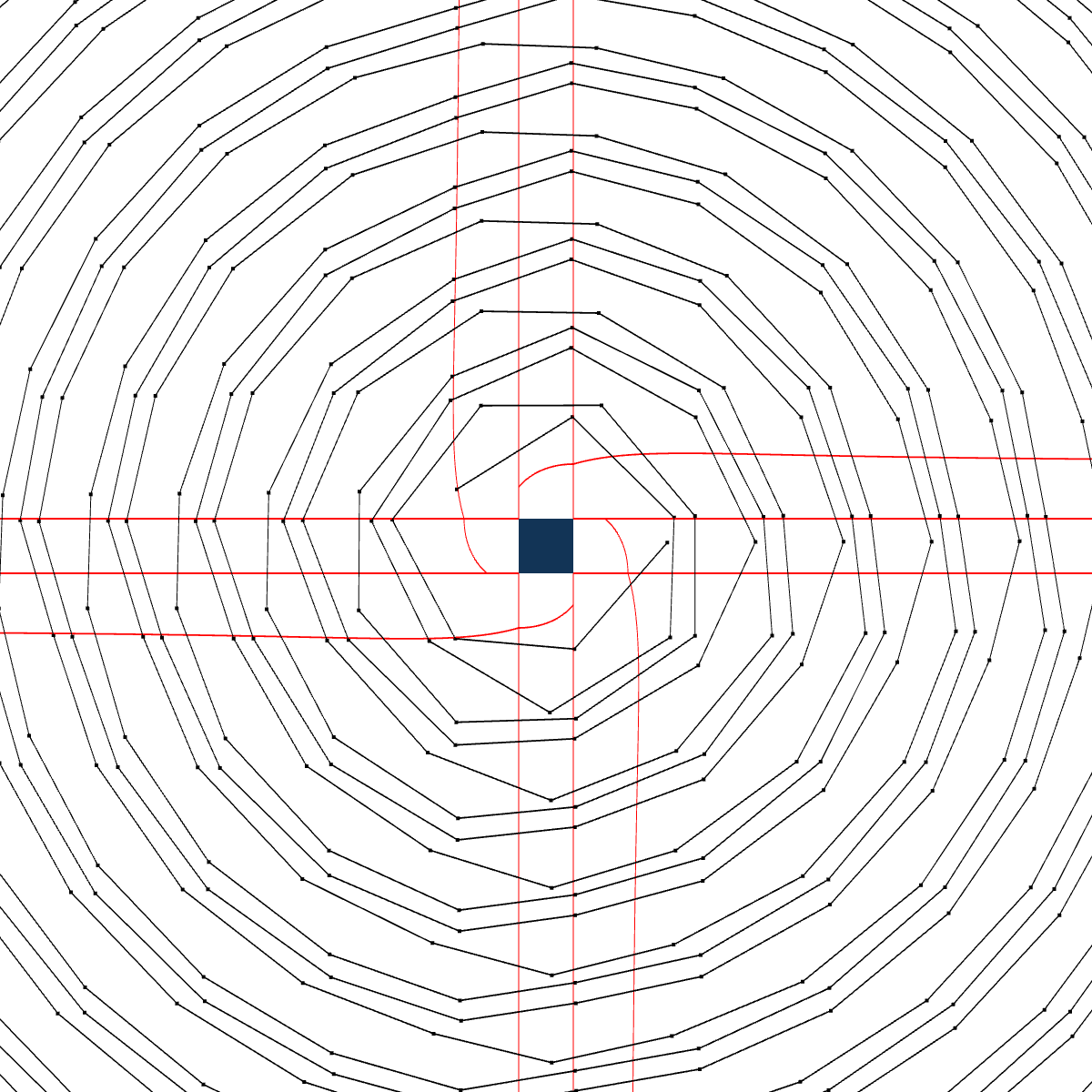}
    \caption{Left: Side extensions of the square and their preimages with a
        labeling of the pieces of the map. Right: a portion of the
        orbit conjecturally escaping ball with line segments connecting each
        point to its second image.}
    \label{fig:square}
\end{figure}
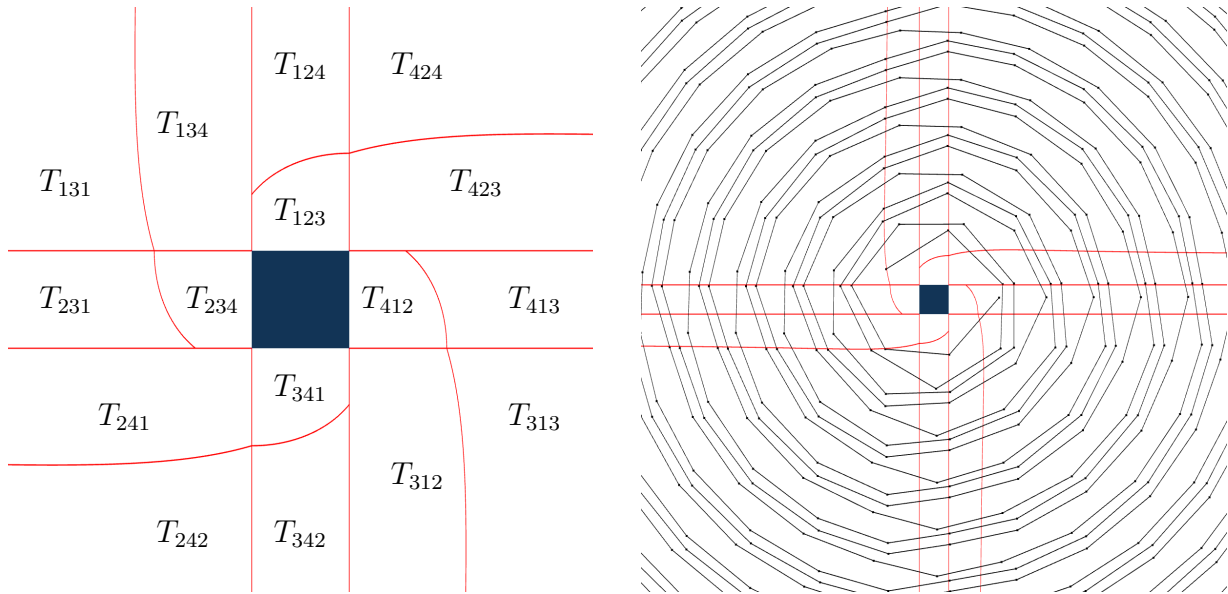

Fix the following notation:
$$E_1:=T_{313},\quad E_2:=T_{424},\quad E_3:=T_{131},\quad E_4:=T_{242}$$
\begin{align*}
    A_{1,n}&:=T_{134}\circ(E_1\circ E_3)^n\circ E_1\circ T_{231}\\
    A_{2,n}&:=T_{241}\circ(E_2\circ E_4)^n\circ E_2\circ T_{342}\\
    A_{4,n}&:=T_{423}\circ(E_3\circ E_1)^n\circ E_3\circ T_{124}\\
    B_{1,n}&:=T_{312}\circ(E_3\circ E_1)^{n+1}\circ T_{231}\\
    B_{2,n}&:=T_{423}\circ(E_4\circ E_2)^{n+1}\circ T_{342}\\
    B_{3,n}&:=T_{134}\circ(E_1\circ E_3)^{n+1}\circ T_{413}\\
    T_n&:=B_{1,n}\circ B_{2,n}\circ B_{3,n}\circ A_{2,n}\circ A_{1,n}\circ
    A_{4,n}.
\end{align*}
    
We expect that there are constants $C_m>0,C_x,C_y$ such that
$$T_n(C_x+nC_m,C_y)=(C_x+(n+1)C_m,C_y)+O(n^{-1}).$$
In other words, we hope to demonstrate an orbit of $T$ with a subsequence which
walks away in a nearly straight line in the positive $x$ direction. The approach
of Dolgopyat and Fayad in \cite{df} appears promising.

\printbibliography

\end{document}